\documentclass[a4paper,12pt,reqno]{amsart}
\usepackage{amsfonts}
\usepackage{amsmath}
\usepackage{amssymb}
\usepackage[a4paper]{geometry}
\usepackage{mathrsfs}
\usepackage{csquotes}
\usepackage[colorlinks]{hyperref}
\renewcommand\eqref[1]{(\ref{#1})} 
\numberwithin{equation}{section}
\theoremstyle{plain}
\newtheorem{thm}{Theorem}[section]

\newtheorem{cor}[thm]{Corollary}
\newtheorem{lem}[thm]{Lemma}
 
\theoremstyle{definition}
\newtheorem{defn}[thm]{Definition}
\newtheorem{rem}[thm]{Remark}

\newcommand{\Rn}{\mathbb R^{n}}

\def\R{\mathcal R}

\def\N{\mathcal N}
\def\e[#1]{{\textrm{e}}^{#1}}

\def\Rn{{\mathbb R}^n}
\def\G{{\mathbb G}}

\def\I{\mathfrak{I}}
\def\L{\mathfrak{L}}

\begin{document}

   \title[Best constants in Sobolev and Gagliardo-Nirenberg inequalities]
   {Best constants in Sobolev and Gagliardo-Nirenberg inequalities on graded groups
   and ground states for higher order nonlinear subelliptic equations}

\author[M. Ruzhansky]{Michael Ruzhansky}
\address{
  Michael Ruzhansky:
  \endgraf
  Department of Mathematics
  \endgraf
  Imperial College London
  \endgraf
  180 Queen's Gate, London SW7 2AZ
  \endgraf
  United Kingdom
  \endgraf
  {\it E-mail address} {\rm m.ruzhansky@imperial.ac.uk}
  }
\author[N. Tokmagambetov]{Niyaz Tokmagambetov}
\address{
  Niyaz Tokmagambetov:
  \endgraf
    al--Farabi Kazakh National University
  \endgraf
  71 al--Farabi ave., Almaty, 050040
  \endgraf
  Kazakhstan,
  \endgraf
   and
  \endgraf
    Department of Mathematics
  \endgraf
  Imperial College London
  \endgraf
  180 Queen's Gate, London, SW7 2AZ
  \endgraf
  United Kingdom
  \endgraf
  {\it E-mail address} {\rm n.tokmagambetov@imperial.ac.uk}
 }
\author[N. Yessirkegenov]{Nurgissa Yessirkegenov}
\address{
  Nurgissa Yessirkegenov:
  \endgraf
  Institute of Mathematics and Mathematical Modelling
  \endgraf
  125 Pushkin str., Almaty, 050010
  \endgraf
  Kazakhstan
  \endgraf
  and
  \endgraf
  Department of Mathematics
  \endgraf
  Imperial College London
  \endgraf
  180 Queen's Gate, London SW7 2AZ
  \endgraf
  United Kingdom
  \endgraf
  {\it E-mail address} {\rm n.yessirkegenov15@imperial.ac.uk}
  }

\thanks{The authors were supported in parts by the EPSRC
 grant EP/K039407/1 and by the Leverhulme Grant RPG-2014-02 as well as by the MESRK (Ministry of Education and Science of the Republic of Kazakhstan) grant 0773/GF4. No new data was collected or
generated during the course of research.}

     \keywords{Nonlinear Schr\"{o}dinger equation, Gagliardo-Nirenberg inequality, Sobolev inequality, graded Lie group, Heisenberg group, Rockland operator, stratified group, sub-Laplacian}
     \subjclass[2010]{35J35, 35G20, 22E30, 43A80}

     \begin{abstract}
In this paper the dependence of the best constants in Sobolev and Gagliardo-Nirenberg inequalities on the precise form of the Sobolev space norm is investigated. The analysis is carried out on general graded Lie groups, thus including the cases of $\mathbb R^n$, Heisenberg, and more general stratified Lie groups. The Sobolev norms may be defined in terms of Rockland operators, i.e. the hypoelliptic homogeneous left-invariant differential operators on the group. The best constants are expressed in the variational form as well as in terms of the ground state solutions of the corresponding nonlinear subelliptic equations. The orders of these equations can be high depending on the Sobolev space order in the Sobolev or Gagliardo-Nirenberg inequalities, or may be fractional. Applications are obtained also to equations with lower order terms given by different hypoelliptic operators. Already in the case of $\Rn$, the obtained results extend the classical relations by Weinstein \cite{W83} to a wide range of nonlinear elliptic equations of high orders with elliptic low order terms and a wide range of interpolation inequalities of Gagliardo-Nirenberg type. However, the proofs are different from those in \cite{W83} because of the impossibility of using the rearrangement inequalities already in the setting of the Heisenberg group. The considered class of graded groups is the most general class of nilpotent Lie groups where one can still consider hypoelliptic homogeneous invariant differential operators and the corresponding subelliptic differential equations.
     \end{abstract}
     
  \maketitle
  
  \tableofcontents

\section{Introduction}
\label{SEC:intro}

The Gagliardo-Nirenberg inequality goes back to works of Gagliardo \cite{G59} and Nirenberg \cite{N59} where it was shown that the inequality
\begin{equation}\label{GN_classic}
\int_{\Rn}|u|^{q}dx\leq C\left(\int_{\Rn}|\nabla u|^{2}dx\right)^{\frac{n(q-2)}{4}}\left(\int_{\Rn}|u|^{2}dx\right)^{\frac{2q-n(q-2)}{4}}
\end{equation}
holds for all $u\in H^{1}(\Rn)$. Here one can take
\begin{equation*}
\left\{\begin{split}
2\leq q \leq \infty \,\,\, &\hbox{for} \,\,\, n=2; \\
2\leq q \leq \frac{2n}{n-2} \,\,\, &\hbox{for} \,\,\, n\geq3.
\end{split}
\right.
\end{equation*}
Weinstein \cite{W83} obtained an expression for the best constant in the inequality \eqref{GN_classic},
relating it to the ground states (least energy solutions) of the nonlinear Schr\"{o}dinger equation
\begin{equation}\label{sem_ell1}
-\Delta u+u=|u|^{q-2}u, \;\;u\in H^{1}(\Rn).
\end{equation}
This result has numerous applications (there are 430 citations to \cite{W83} on MathSciNet), for example to further properties of the critical nonlinear Schr\"{o}dinger equation \eqref{sem_ell1}, \cite{MR04, MR05}, and to many other problems.

On the Heisenberg group $\mathbb{H}^{N}$, the subelliptic Gagliardo-Nirenberg inequality takes the form
\begin{equation}\label{GN_Heisenberg}
\int_{\mathbb{H}^{N}}|u|^{q}dx\leq C\left(\int_{\mathbb{H}^{N}}|\nabla_{H} u|^{2}dx\right)^{\frac{Q(q-2)}{4}}\left(\int_{\mathbb{H}^{N}}|u|^{2}dx\right)^{\frac{2q-Q(q-2)}{4}},
\end{equation}
where $\nabla_{H}$ is a horizontal gradient, $Q=2N+2$ is the homogeneous dimension of $ \mathbb{H}^{N}$, $2<q<2+\frac{2}{N}$.
In \cite{CR13}, the best constant for the Gagliardo-Nirenberg inequality \eqref{GN_Heisenberg} on the Heisenberg group was expressed in terms of the ground state solutions of the subelliptic equation
\begin{equation}\label{sem_ell2}
-\triangle_{H}u+u=|u|^{q-2}u, \;\;u\in H^1(\mathbb{H}^{N}),
\end{equation}
where $\triangle_{H}$ is the sub-Laplacian on $\mathbb{H}^{N}$, and $H^1(\mathbb{H}^{N})$ is the Sobolev space on $\mathbb{H}^{N}$ with the norm
$$\|u\|:=\left(\int_{\mathbb{H}^{N}}(|\nabla_{H}u|^{2}+|u|^{2})dx\right)^{1/2}.$$

One of the aims of this paper is to answer the following questions:

\begin{itemize}
\item How do the best constants in the Gagliardo-Nirenberg inequalities \eqref{GN_classic}, \eqref{GN_Heisenberg} depend on the precise formula for the Sobolev norms? For example, if we replace $\|u\|_{\dot{H}^1(\Rn)}=\|\nabla u\|_{L^2(\Rn)}$ by the equivalent norm $\|(-\Delta )^{1/2}u\|_{L^2(\Rn)}$ or by
the equivalent norms $\|((-1)^m \sum_{j=1}^n \partial_{x_j}^{2m})^{\frac1{2m}}u\|_{L^2(\Rn)}$, and similarly for the Heisenberg group, how does it influence the best constants in \eqref{GN_classic}, \eqref{GN_Heisenberg} and the nonlinear equations \eqref{sem_ell1}, \eqref{sem_ell2}?
\item What can be said about more general Gagliardo-Nirenberg inequalities? For example, when the first order Sobolev norm in \eqref{GN_classic}, \eqref{GN_Heisenberg} is replaced by higher order Sobolev norms? Also, when the $L^2$-norms on the right hand sides in \eqref{GN_classic}, \eqref{GN_Heisenberg} are replaced by appropriate $L^p$-norm for other values of $p$?
\item What about nonlinear equations composed of several (hypo)elliptic terms, for example,
\begin{equation}\label{sem_ell2h}
\sum_{j=1}^n\frac{\partial^4 u}{\partial x_j^4} -\triangle u+u=|u|^{q-2}u, \;\;u\in H^2(\mathbb{R}^{n}),
\end{equation}
instead of \eqref{sem_ell1} or, for example,
\begin{equation}\label{sem_ell2h2}
\sum_{j=1}^N(X_j^8+Y_j^8)u +\triangle_{H}^2u-\triangle_Hu+u=|u|^{q-2}u, \;\;u\in H^4(\mathbb{H}^{N}),
\end{equation}
instead of \eqref{sem_ell2}, where the vectors $\{X_j,Y_j\}_{j=1}^N$ give the first stratum of $\mathbb{H}^{N}$. What can be said about ground state solutions for such equations?
\item
{\em From this perspective, our results will be new already in the classical Euclidean setting of $\Rn$,} for example also allowing one to consider differential equations on $\Rn$ of the form
\begin{equation}\label{EQ:hoex0}
\sum_{j=1}^\ell \left( (-1)^{m_j}\sum_{k=1}^n a_{jk}\frac{\partial^{2m_j} u}{\partial x_k^{2m_j}}\right)=|u|^{q-2}u, \quad a_{jk}>0,\; m_j\in\mathbb N_0,
\end{equation}
proving the existence of ground state (least energy) solutions for such equations.
\end{itemize}

A natural setting for our analysis will be that of graded Lie groups as developed by Folland and Stein \cite{FS-book}. This is the largest class of homogeneous nilpotent Lie groups admitting homogeneous hypoelliptic left-invariant differential operators (\cite{Miller:80}, \cite{tER:97}, see also a discussion in \cite[Section 4.1]{FR16}).
These operators are called Rockland operator, after Helffer and Nourrigat's resolution \cite{HN-79} of the Rockland conjecture. Thus, our setting will include the higher order operators on $\Rn$ as well as higher order hypoelliptic invariant differential operators on the Heisenberg group, on general stratified groups, and on general graded Lie groups. We also note that the Rockland operators on graded Lie groups appear naturally in the analysis of subelliptic operators on manifolds, starting with the seminal paper of Rothschild and Stein \cite{Rothschild-Stein:AM-1976}.

Thus, let $\G$ be a graded Lie group, i.e. a connected simply connected Lie group such that its Lie algebra $\mathfrak{g}$ admits a decomposition
$$\mathfrak{g}=\bigoplus_{\ell=1}^{\infty}\mathfrak{g}_{\ell},$$
where the $\mathfrak{g}_{\ell}$, $\ell=1,2,...,$ are vector subspaces of $\mathfrak{g}$, all but finitely many equal to $\{0\}$, and satisfying
$$[\mathfrak{g}_{\ell},\mathfrak{g}_{\ell'}]\subset \mathfrak{g}_{\ell+\ell'} \;\;\forall \ell, \ell'\in \mathbb{N}.$$
Such groups are then necessarily nilpotent and homogeneous, they can be identified with $\Rn$ through the exponential mapping, with $n$ being the topological dimension of $\G$, with the Haar measure on $\G$ given by the Lebesgue measure on $\Rn$ through this identification.
A family of dilations of a Lie algebra $\mathfrak{g}$ is a family of linear mappings of the form
$$D_{r}={\rm Exp}(A \,{\rm ln}r)=\sum_{k=0}^{\infty}
\frac{1}{k!}({\rm ln}(r) A)^{k},$$
where $A$ is a diagonalisable linear operator on the Lie algebra $\mathfrak{g}$ with positive eigenvalues, and each $D_{r}$ is a morphism of $\mathfrak{g}$,
that is, a linear mapping from $\mathfrak{g}$ to itself satisfying
$$\forall X,Y\in \mathfrak{g},\, r>0,\;
[D_{r}X, D_{r}Y]=D_{r}[X,Y],$$
where $[X,Y]:=XY-YX$ is the Lie bracket. If $\nu_{1},\ldots,\nu_{n}$ are weights of the dilations, i.e. the eigenvalues of the matrix $A$, then the group's dilations are defined through the exponential mapping by
$$D_{r}(x)=rx:=(r^{\nu_{1}}x_{1},\ldots,r^{\nu_{n}}x_{n}), \;\;x=(x_{1},\ldots,x_{n})\in\mathbb{G},\;\;r>0.$$
The homogeneous dimension of $\G$ is defined by
$$
Q:={\rm Tr}\, A=\nu_1+\cdots+\nu_n.
$$
In Section \ref{SEC:prelim} we will give a short overview of graded Lie groups, but we can mention here that we have $\Rn$ as a special case with all $\nu_j=1$, or the stratified groups when $\mathfrak g$ is generated by its first stratum $\mathfrak g_1$ through the iterative application of commutators.

Let $\R$ be a {\em positive Rockland operator on $\G$, i.e. a homogeneous hypoelliptic left-invariant differential operator, positive in the operator sense}. Such operators always exist. For example, for the Heisenberg group, the sub-Laplacian and its powers are Rockland operators.
If $\G$ is a stratified Lie group with a basis  $X_1,\ldots,X_k\in\mathfrak g_1$ of $\mathfrak g_1$, the operators
$$
\R=(-1)^m\sum_{j=1}^k a_j  X_j^{2m},\quad
a_j>0,
$$
are positive Rockland operators for any $m\in\mathbb N$, yielding the sub-Laplacian for $m=1$.
More generally, for any graded Lie group $\G\sim\Rn$ with dilation weights $\nu_1,\ldots,\nu_n$ and a basis $X_1,\ldots,X_n$ of the Lie algebra $\mathfrak g$ of $\G$ satisfying
$$
D_r X_j=r^{\nu_j} X_j,\quad j=1,\ldots,n,\; r>0,
$$
the operator
$$
\R=\sum_{j=1}^n (-1)^{\frac{\nu_0}{\nu_j}} a_j X_j^{2\frac{\nu_0}{\nu_j}},\quad a_j>0,
$$
is a Rockland operator of homogeneous degree $2\nu_0$, if $\nu_0$ is any common multiple of $\nu_1,\ldots,\nu_n$. There are other examples of Rockland operators that can be adapted to special selections of vector fields generating the Lie algebra in special ways, such as for example the vector fields from the first stratum on the stratified Lie groups. We refer to \cite[Section 4.1.2]{FR16} for other examples and a detailed discussion of Rockland operators.

Sobolev spaces associated to positive Rockland operators on graded Lie groups have been analysed in \cite{FR:Sobolev} and in \cite[Section 4.4]{FR16}. In particular, for a positive Rockland operator of homogeneous degree $\nu\in\mathbb N$, for $a>0$ and $1\leq p<\infty$, we can define the homogeneous and inhomogeneous Sobolev spaces, respectively, by the norms
\begin{equation}\label{EQ:Sob-norms}
\|u\|_{\dot{L}^p_{a,\R}(\G)}:=\|\R^{\frac{a}\nu}u\|_{L^p(\G)} \;\textrm{ and }\;
\|u\|_{{L}^p_{a,\R}(\G)}:=(\|\R^{\frac{a}\nu}u\|_{L^p(\G)}^p+\|u\|_{L^p(\G)}^p)^{\frac1p}.
\end{equation}
If $\G=\Rn$ and $\R$ is a homogeneous elliptic operator with constant coefficients, these spaces coincide with the usual Sobolev spaces on $\Rn$. If $\G$ is a stratified Lie group and $\R$ is a positive sub-Laplacian, then these Sobolev spaces have been analysed by Folland in \cite{F75}.
For an extensive analysis of these Sobolev spaces on general graded Lie group we refer to \cite{FR:Sobolev} or \cite[Section 4.4]{FR16}. In particular, it was shown that these spaces are independent of the choice of a positive operator $\R$, so that we can drop the subscript $\R$ in their notation, abbreviating them to $\dot{L}^p_a$ and ${L}^p_a$, respectively.
While these spaces are independent of $\R$, the particular norms in \eqref{EQ:Sob-norms} do depend on it.

The starting point of our analysis in this paper is the following Gagliardo-Nirenberg-Sobolev inequality on graded groups:

\begin{itemize}
\item {\bf (Gagliardo-Nirenberg inequality)}
Let $\mathbb{G}$ be a graded Lie group of homogeneous dimension $Q$ and let $\mathcal{R}_{1}$ and $\mathcal{R}_{2}$ be positive Rockland operators of homogeneous degrees $\nu_{1}$ and $\nu_{2}$, respectively. Let $a_{1}> a_{2}\geq0$, $1<p<\frac{Q}{a_{1}}$ and $\frac{pQ}{Q-a_{2}p}\leq q\leq\frac{pQ}{Q-a_{1}p}$. Then there exists a constant $C>0$ such that
\begin{equation}\label{EQ:GN1}
\int_{\mathbb{G}}|u(x)|^{q}dx\leq C \left(\int_{\mathbb{G}}|\mathcal{R}_{1}^{\frac{a_{1}}{\nu_{1}}}u(x)|^{p}dx\right)^{\frac{Q(q-p)-a_{2}pq}{(a_{1}-a_{2})p^{2}}}
\left(\int_{\mathbb{G}}|\mathcal{R}_{2}^{\frac{a_{2}}{\nu_{2}}}u(x)|^{p}dx\right)^{\frac{a_{1}pq-Q(q-p)}{(a_{1}-a_{2})p^{2}}}
\end{equation}
holds for all $u\in \dot{L}^{p}_{a_{1}}(\mathbb{G})\cap \dot{L}^{p}_{a_{2}}(\mathbb{G})$.
\end{itemize}

The freedom of working with two different Rockland operators $\R_1, \R_2$ in \eqref{EQ:GN1} leads to the possibility of considering two different (hypo)elliptic operators in equations
\eqref{sem_ell2h} and \eqref{sem_ell2h2}. Throughout this paper we will often abbreviate the notation by writing
$$
L^p_{a_1,a_2}(\mathbb{G}):=\dot{L}^{p}_{a_{1}}(\mathbb{G})\cap \dot{L}^{p}_{a_{2}}(\mathbb{G}).
$$
The inequality \eqref{EQ:GN1} will be established in Section \ref{SEC:GN_ineq}.
Consequently, the question arises of what is the best constant $C$ in this inequality, which we may denote by $C_{GN,\R_{1},\R_{2}}=
C_{GN,\R_{1},\R_{2},a_{1},a_{2},p,q}$ since it depends on the operators $\R_{1}, \R_{2}$ as well as on the indices $a_{1},a_{2},p,q$. The related question is of the best constant in the Sobolev (embedding) inequality
\begin{equation}\label{Sobolev1}
\left(\int_{\G}|u(x)|^{q}dx\right)^{\frac{p}{q}}\leq C\int_{\G}(|\mathcal{R}^{\frac{a}{\nu}}u(x)|^{p}+|u(x)|^{p})dx,
\end{equation}
where $u\in L^{p}_{a}(\mathbb{G})$, see \cite[Theorem 4.4.28]{FR16} for its proof in the setting of general graded Lie groups.

In this paper we will show that both the Sobolev inequality \eqref{Sobolev1} and the Gagliardo-Nirenberg inequality \eqref{EQ:GN1} are related to the
following Schr\"{o}dinger equation with the power nonlinearities:
\begin{equation}\label{nonlinear1}
\mathcal{R}_{1}^{\frac{a_{1}}{\nu_{1}}}(|\mathcal{R}_{1}^{\frac{a_{1}}{\nu_{1}}}u|^{p-2}\mathcal{R}_{1}^{\frac{a_{1}}{\nu_{1}}}u)+
\mathcal{R}_{2}^{\frac{a_{2}}{\nu_{2}}}(|\mathcal{R}_{2}^{\frac{a_{2}}{\nu_{2}}}u|^{p-2}\mathcal{R}_{2}^{\frac{a_{2}}{\nu_{2}}}u)=|u|^{q-2}u.
\end{equation}
Examples of such equation are given by equations \eqref{sem_ell2h} and \eqref{sem_ell2h2}.
Moreover, these inequalities are related to the variational problem
\begin{equation}\label{di}
d=\inf_{\stackrel{u\in L^{p}_{a_{1},a_{2}}(\mathbb{G})\backslash\{0\}}{\I(u)=0}} \L(u),
\end{equation}
for functionals
\begin{equation}\label{Li}
\mathfrak{L}(u)=\frac{1}{p}\int\limits_{\mathbb{G}}|\mathcal{R}_{1}^{\frac{a_{1}}{\nu_{1}}}u(x)|^{p}dx+
\frac{1}{p}\int\limits_{\mathbb{G}}|\mathcal{R}_{2}^{\frac{a_{2}}{\nu_{2}}}u(x)|^{p}dx
-\frac{1}{q}\int\limits_{\mathbb{G}}|u(x)|^{q}dx
\end{equation}
and
\begin{equation}\label{Ii}
\mathfrak{I}(u):=\int\limits_{\mathbb{G}}(|\mathcal{R}_{1}^{\frac{a_{1}}{\nu_{1}}}u(x)|^{p}+
|\mathcal{R}_{2}^{\frac{a_{2}}{\nu_{2}}}u(x)|^{p}-|u(x)|^{q})dx.
\end{equation}

Thus, in this paper we will show that
\begin{itemize}
\item {\bf (Existence of ground state solutions)} Let $a_{1}> a_{2}\geq0$, $1<p<\frac{Q}{a_{1}}$ and $\frac{pQ}{Q-a_{2}p}<q<\frac{pQ}{Q-a_{1}p}$. Then the nonlinear Schr\"{o}dinger type equation \eqref{nonlinear1} has a least energy solution $\phi\in
\dot{L}^{p}_{a_{1}}(\mathbb{G})\cap \dot{L}^{p}_{a_{2}}(\mathbb{G})$, i.e. a solution $\phi$ such that $d=\L(\phi)$.

\item {\bf (Best constants in Sobolev embeddings)}
Let $a>0$, $1<p<\frac{Q}{a}$, $p<q<\frac{pQ}{Q-ap}$, let $\phi$ be a least energy solution of \eqref{nonlinear1}, and  let $d=\L(\phi)$.
Let $C_{S,\R, a,p,q}$ be the best constant in the Sobolev embedding theorem, i.e. the smallest constant $C$ in the inequality \eqref{Sobolev1}.
Then we have
\begin{equation}\label{EQ:Si}
C_{S, \R,a,p,q}=\left(\frac{apq}{apq-Q(q-p)}\int_{\G}|\phi(x)|^{p}dx\right)^{\frac{p-q}{q}}=
\left(\frac{pq}{q-p}d\right)^{\frac{p-q}{q}}.
\end{equation}

\item {\bf (Best constants in Gagliardo-Nirenberg inequalities)}
Let $a_{1}> a_{2}\geq0$, $1<p<\frac{Q}{a_{1}}$, $\frac{pQ}{Q-a_{2}p}<q<\frac{pQ}{Q-a_{1}p}$, let $\phi$ be a least energy solution of \eqref{nonlinear1}, and  let $d=\L(\phi)$.
Let $C_{GN,\R_{1},\R_{2},a_{1},a_{2},p,q}$ be the best constant in the Gagliardo-Nirenberg inequality, i.e. the smallest constant $C$ in the inequality \eqref{GN1}. Then we have
$$C_{GN,\R_{1},\R_{2},a_{1},a_{2},p,q}$$
\begin{equation}\label{EQ:GNi}=\frac{(a_{1}-a_{2})pq}{a_{1}pq-Q(q-p)}
\left(\frac{Q(q-p)-a_{2}pq}{a_{1}pq-Q(q-p)}\right)^{\frac{a_{2}pq-Q(q-p)}{(a_{1}-a_{2})p^{2}}}
\|\mathcal{R}_{2}^{\frac{a_{2}}{\nu_{2}}}\phi\|_{L^{p}(\mathbb{G})}^{p-q}
\end{equation}
$$ =
\frac{(a_{1}-a_{2})pq}{a_{1}pq-Q(q-p)}
\left(\frac{Q(q-p)-a_{2}pq}{a_{1}pq-Q(q-p)}\right)^{\frac{a_{2}pq-Q(q-p)}{(a_{1}-a_{2})p^{2}}}
\left(\frac{a_{1}pq-Q(q-p)}{(a_{1}-a_{2})(q-p)}d\right)^{\frac{p-q}{p}}.
$$
\end{itemize}
The obtained results provide new insights already in the case of $\G=\Rn$ in view of the arbitrariness of the operators $\R$ and $\R_{1}$, $\R_{2}$, which in this case may be any homogeneous elliptic differential operator with constant coefficients, and of any order. Moreover, in this case the proof works equally well if it is a pseudo-differential operator. We note that on $\Rn$ for some indices explicit expressions for best constants in Sobolev inequalities are available, see e.g. \cite{Aubin76, Talenti76}, and also \cite{LL}.

Let us give some  new examples but, more interestingly, in the setting of the Heisenberg group $\mathbb H^N$ of homogeneous dimension $Q=2N+2$. Taking $\R_{1}=\R_{2}=-\Delta_H$ the sub-Laplacian on $\mathbb H^N$, $\nu=2$, for $p=2$ and $a_{1},a_{2},m_{1}\in\mathbb N$, $m_{2}\in\mathbb N\cup\{0\}$ with $a_{1}=m_{1}>a_{2}=m_{2}$, the equation \eqref{nonlinear1} becomes a nonlinear differential equation for the poly-sub-Laplacian
\begin{equation}\label{nonlinear1ex}
(-\Delta_H)^{m_{1}}u+(-\Delta_H)^{m_{2}}u=|u|^{q-2}u.
\end{equation}
Thus, it follows from the results of this paper that for $Q>2m_{1}$ and $\frac{2Q}{Q-m_{2}p}< q <\frac{2Q}{Q-m_{1}p}$, this
equation has a ground state $\phi\in L^2_{m_{1},m_{2}}$ minimising the variational problem \eqref{di},
such that $d=\L(\phi)$ enters the expressions \eqref{EQ:Si} and \eqref{EQ:GNi} for the best constants in the Sobolev inequality
$$
\|u\|_{L^{q}(\G)}^2\leq C(\|(-\Delta_H)^{\frac{m_{1}}{2}}u\|_{L^{2}(\G)}^2+\|(-\Delta_H)^{\frac{m_{2}}{2}}u\|_{L^{2}(\G)}^2)
$$
and in the Gagliardo-Nirenberg inequality
\begin{multline}\label{EQ:GN1i}
\int_{\mathbb{G}}|u(x)|^{q}dx\leq \\
C
\left(\int_{\G}|(-\Delta_H)^{\frac{m_{1}}{2}}u(x)|^{2}dx\right)^{\frac{Q(q-2)-2m_{2}q}{4(m_{1}-m_{2})}}
\left(\int_{\G}|(-\Delta_H)^{\frac{m_{2}}{2}}u(x)|^{2}dx\right)^{\frac{2m_{1}q-Q(q-2)}{4(m_{1}-m_{2})}},
\end{multline}
respectively. The same is true if we replace the Heisenberg group by any stratified group, or if we replace it by $\Rn$ also replacing $\Delta_H$ by the Laplacian and $Q$ by $n$.

Some boundary value problems for the poly-sub-Laplacians $(-\Delta_H)^{m}$ have been studied in \cite{RS-PAMS} on the Heisenberg group, and in \cite{RS-AM} on general stratified groups.

The possibility of using different Rockland operators at the same time leads to a wider variety of equations where our results are applicable. For example, let $\{X_j,Y_j\}_{j=1}^N$ be the basis of the first stratum of a stratified Lie group $\mathbb G$. So, this is the case of the Heisenberg group $\mathbb H^N$, or of $\mathbb G=\mathbb R^n$ with $N=n$.
Let us now take
$$\R_{1}=(-1)^{m_1}\sum_{j=1}^N (X_j^{2m_1}+Y_j^{2m_1}),\quad \R_{2}=(-1)^{m_2} \sum_{j=1}^N \alpha_j(X_j^{2m_2}+Y_j^{2m_2})
$$
for some $\alpha_j>0$, $m_1,m_2\in\mathbb N_0$, $m_2>m_1$, so that $\nu_1=2m_1$, $\nu_2=2m_2$.
For $p=2$ the equation \eqref{nonlinear1} becomes a nonlinear differential equation \begin{equation}\label{nonlinear1ex2}
(-1)^{m_1}\sum_{j=1}^N (X_j^{2m_1}+Y_j^{2m_1}) u+(-1)^{m_2}\sum_{j=1}^N \alpha_j(X_j^{2m_2}+Y_j^{2m_2})u=|u|^{q-2}u.
\end{equation}
Thus, it follows from the results of this paper that for $Q>2m_{1}$ and $\frac{2Q}{Q-m_{2}p}< q <\frac{2Q}{Q-m_{1}p}$, this
equation has a ground state $\phi\in L^2_{m_{1},m_{2}}$ minimising the variational problem \eqref{di},
such that $d=\L(\phi)$ enters the expressions \eqref{EQ:Si} and \eqref{EQ:GNi} for the best constants in the Sobolev inequality
$$
\|u\|_{L^{q}(\G)}^2\leq C(\|\sum_{j=1}^N (X_j^{2m_1}+Y_j^{2m_1})u\|_{L^{2}(\G)}^2+\|\sum_{j=1}^N \alpha_j (X_j^{2m_2}+Y_j^{2m_2})u\|_{L^{2}(\G)}^2)
$$
and in the Gagliardo-Nirenberg inequality
\begin{multline}\label{EQ:GN1i2}
\int_{\mathbb{G}}|u(x)|^{q}dx\leq
C
\left(\int_{\G}|\sum_{j=1}^N (X_j^{2m_1}+Y_j^{2m_1}) u(x)|^{2}dx\right)^{\frac{Q(q-2)-2m_{2}q}{4(m_{1}-m_{2})}} \\
\times
\left(\int_{\G}|\sum_{j=1}^N \alpha_j(X_j^{2m_2}+Y_j^{2m_2}) u(x)|^{2}dx\right)^{\frac{2m_{1}q-Q(q-2)}{4(m_{1}-m_{2})}},
\end{multline}
respectively.

Weinstein \cite{W83} gave the proof of these results in the case $\G=\Rn$, $\R=\R_{1}$ corresponding to the gradient, $p=2$ and $a_{1}=1$, $a_{2}=0$, by finding the solution to the minimisation problem \eqref{di}. However, his techniques rely on the rearrangement inequalities which are, therefore, specific to $\Rn$ in the setting of general nilpotent Lie groups. In the case of $\G$ being the Heisenberg group, $\R=\R_{1}$ corresponding to the horizontal gradient, $p=2$ and $a_{1}=1$, $a_{2}=0$, these results have been obtained in \cite{CR13} by a different method relying on obtaining upper and lower estimates on the best constants in the Gagliardo-Nirenberg inequality. This method was effective in several other problems, for example in weighted nonlinear equations \cite{CR-Opuscula}.
In this paper we will also use this method but now extending it to larger ranges of indices and to general graded Lie groups.

Similar results have been investigated for Riemannian manifolds and hyperbolic spaces, see for example \cite{MS13} and \cite{HV95}.

The paper is structured as follows. In Section \ref{SEC:prelim} we briefly recall further main concepts of graded Lie groups and fix the notation.
The Gagliardo-Nirenberg inequality on graded Lie group is established in Section \ref{SEC:GN_ineq}. In Section \ref{SEC:exist} we show the existence of least energy solutions of subelliptic equation \eqref{nonlinear1}. The expressions of the best constant in the Gagliardo-Nirenberg and Sobolev inequalities are obtained in Section \ref{SEC:sharp} and in Section \ref{SEC:sharp2}, respectively. Finally, in Section \ref{SEC:ex_exist} we prove the more general form of the Gagliardo-Nirenberg inequality and present the extension of the main result of the Section \ref{SEC:exist}.

\section{Preliminaries}
\label{SEC:prelim}

In this section we very briefly recall the necessary notation concerning the setting of graded groups. For a detailed
description of the notions of graded and homogeneous nilpotent Lie groups we refer to Folland and Stein \cite[Chapter 1]{FS-book}, or to the recent exposition in \cite[Chapter 3]{FR16}.

A connected simply connected Lie group $\mathbb{G}$ is called a graded Lie group if its Lie algebra $\mathfrak{g}$ admits a gradation
$$\mathfrak{g}=\bigoplus_{\ell=1}^{\infty}\mathfrak{g}_{\ell},$$
where the $\mathfrak{g}_{\ell}$, $\ell=1,2,...,$ are vector subspaces of $\mathfrak{g}$, all but finitely many equal to $\{0\}$, and satisfying
$$[\mathfrak{g}_{\ell},\mathfrak{g}_{\ell'}]\subset \mathfrak{g}_{\ell+\ell'} \;\;\forall \ell, \ell'\in \mathbb{N}.$$

Let us now fix a basis $\{X_{1},\ldots,X_{n}\}$ of a Lie algebra $\mathfrak{g}$ adapted to the gradation. We obtain points in $\mathbb{G}$ by the exponential mapping $\exp_{\mathbb{G}}:\mathfrak{g}\rightarrow\mathbb{G}$ as
$$x=\exp_{\mathbb{G}}(x_{1}X_{1}+\ldots+x_{n}X_{n}).$$
A family of dilations of a Lie algebra $\mathfrak{g}$ is a family of linear mappings of the following form
$$D_{r}={\rm Exp}(A \,{\rm ln}r)=\sum_{k=0}^{\infty}
\frac{1}{k!}({\rm ln}(r) A)^{k},$$
where $A$ is a diagonalisable linear operator on the Lie algebra $\mathfrak{g}$ with positive eigenvalues, and each $D_{r}$ is a morphism of $\mathfrak{g}$,
that is, a linear mapping from $\mathfrak{g}$ to itself satisfying:
$$\forall X,Y\in \mathfrak{g},\, r>0,\;
[D_{r}X, D_{r}Y]=D_{r}[X,Y],$$
where $[X,Y]:=XY-YX$ is the Lie bracket. The dilations can be extended through the exponential mapping to the group $\G$ by
$$D_{r}(x)=rx:=(r^{\nu_{1}}x_{1},\ldots,r^{\nu_{n}}x_{n}), \;\;x=(x_{1},\ldots,x_{n})\in\mathbb{G},\;\;r>0,$$
where $\nu_{1},\ldots,\nu_{n}$ are weights of the dilations. The homogeneous dimension of $\G$ is defined by
$$
Q:={\rm Tr}\, A=\nu_1+\cdots+\nu_n.
$$

We note that the standard Lebesgue measure $dx$ on $\mathbb{R}^{n}$ is the Haar measure for $\mathbb{G}$ (see, e.g. \cite[Proposition 1.6.6]{FR16}). We also recall that a {homogeneous quasi-norm} on $\mathbb G$ is
a continuous non-negative function
$$\mathbb{G}\ni x\mapsto |x|\in [0,\infty)$$
which satisfies the following properties
\begin{itemize}
\item   $|x^{-1}| = |x|$ for all $x\in \mathbb{G}$,
\item  $|\lambda x|=\lambda |x|$ for all
$x\in \mathbb{G}$ and $\lambda >0$,
\item  $|x|= 0$ if and only if $x=0$.
\end{itemize}

The quasi-ball centred at $x\in\mathbb{G}$ with radius $R > 0$ is defined by
$$B(x,R):=\{y\in \mathbb{G}: |x^{-1}y|<R\}.$$

Let $\widehat{\mathbb{G}}$ denote the unitary dual of $\mathbb{G}$.
For a representation $\pi\in\widehat{\mathbb{G}}$, let $\mathcal{H}_{\pi}^{\infty}$ denote the space of smooth vectors for it. Then a Rockland operator $\mathcal{R}$ on $\mathbb{G}$ is a left-invariant differential operator which is homogeneous of positive degree and satisfies the Rockland condition:

({\bf R}) for every representation $\pi\in\widehat{\mathbb{G}}$, except for the trivial representation, the operator $\pi(\R)$ is injective on $\mathcal{H}_{\pi}^{\infty}$, that is,
$$\forall \upsilon \in \mathcal{H}_{\pi}^{\infty}, \;\;\pi(\R)\upsilon=0\Rightarrow \upsilon=0,$$
where $\pi(\R):=d\pi(\R)$ is the infinitesimal representation of $\R$ as of an element of the universal enveloping algebra of $\G$.

We refer to \cite[Definition 1.7.4 and Section 4.1.1]{FR16} for a detailed discussion of this definition, that appeared in the work of Rockland \cite{Rockland}. Different characterisations of such operators have been obtained by Rockland \cite{Rockland} and Beals \cite{Beals-Rockland}, until the resolution in \cite{HN-79} by Helffer and Nourrigat of the so-called Rockland conjecture, which characterised operators satisfying condition ({\bf R}) as left-invariant homogeneous hypoelliptic differential operators on $\G$.

In this paper we will not be using the representation theoretic interpretation of these operators, so we define
{\em Rockland operators as left-invariant homogeneous hypoelliptic differential operators on $\G$.}

Moreover, in this paper we will deal with the Rockland differential operators which are positive in the sense of operators.

We refer to \cite[Chapter 4]{FR16} for an extensive presentation concerning Rockland operators and their properties, as well as for the consistent development of the corresponding theory of Sobolev spaces. The corresponding Besov spaces on graded Lie groups and their properties appeared in \cite{CR-CRAS}. A different version of the Gagliardo-Nirenberg inequality on graded Lie groups appeared in \cite{BFKG-graded}.
The pseudo-differential calculus on graded Lie groups appeared in \cite{FR:graded} and then in \cite{FR16}. Spectral properties of the infinitesimal representations of Rockland operators have been analysed in \cite{tER:97}.

\section{Gagliardo-Nirenberg inequality on graded Lie group}
\label{SEC:GN_ineq}

Let $\mathbb{G}$ be a graded Lie group of homogeneous dimension $Q$ and let $\mathcal{R}_{1}$ and $\mathcal{R}_{2}$ be positive Rockland operators of homogeneous degree $\nu_{1}$ and $\nu_{2}$, respectively.

In this section we investigate the Gagliardo-Nirenberg inequality on graded Lie groups. We denote the Sobolev space by $L^{p}_{a}(\mathbb{G})=L^{p}_{a,\R}(\mathbb{G})$, for $a>0$, defined by the norm
\begin{equation}\label{norm_0}
\|u\|_{L^{p}_{a, \R}(\mathbb{G})}:=\left(\int_{\mathbb{G}}(|\mathcal{R}^{\frac{a}{\nu}}u(x)|^{p}+|u(x)|^{p})dx\right)^{1/p}.
\end{equation}
Let us denote
$$\|u\|_{\dot{L}_{a,\R}^{p}(\G)}:=\|\mathcal{R}^{a/\nu}u\|_{L^{p}(\G)}.$$
We also use the space $L^{p}_{a_{1},a_{2}}(\mathbb{G})=L^{p}_{a_{1}, a_{2}, \mathcal{R}_{1}, \R_{2}}(\mathbb{G})$, for $a_{1}>a_{2}\geq0$, defined by the norm
\begin{equation}\label{norm}
\|u\|_{L^{p}_{a_{1}, a_{2}, \R_{1}, \R_{2}}(\mathbb{G})}:=\left(\int_{\mathbb{G}}(|\mathcal{R}_{1}^{\frac{a_{1}}{\nu_{1}}}u(x)|^{p}+
|\mathcal{R}_{2}^{\frac{a_{2}}{\nu_{2}}}u(x)|^{p})dx\right)^{1/p}.
\end{equation}
\begin{rem}\label{rem_ind}
We refer to \cite[Theorem 4.4.20]{FR16} for the independence of the spaces $L^{p}_{a}(\mathbb{G})$ of a particular choice of the Rockland operator $\R$. Consequently, the spaces $L^{p}_{a_{1}, a_{2}}(\G)=\dot{L}^{p}_{a_{1}}(\mathbb{G})\cap \dot{L}^{p}_{a_{2}}(\mathbb{G})$ are also independent of the choice of Rockland operators $\R_{1}$ and $\R_{2}$. Note that if $a_{2}=0$, then $L^{p}_{a_{1},0}=L^{p}_{a_{1}}(\G)$. Let us show this independence and another relation also for $a_{1}\neq a_{2}$. We know that
$$\|\mathcal{R}_{2}^{\frac{a_{2}}{\nu_{2}}}f\|_{L^{p}(\G)}<\infty \Longleftrightarrow \|\mathcal{R}_{1}^{\frac{a_{2}}{\nu_{1}}}f\|_{L^{p}(\G)}<\infty.$$
Denoting $\mathcal{R}_{1}^{\frac{a_{2}}{\nu_{1}}}f=g$, we get
$$\|\mathcal{R}_{1}^{\frac{a_{1}}{\nu_{1}}}f\|_{L^{p}(\G)}=\|\mathcal{R}_{1}^{\frac{a_{1}-a_{2}}{\nu_{1}}}g\|_{L^{p}(\G)}<\infty,$$
which gives $g\in \dot{L}^{p}_{a_{1}-a_{2}}(\G)$. It follows that $f\in\mathcal{R}_{2}^{-\frac{a_{2}}{\nu_{1}}}(\dot{L}^{p}_{a_{1}-a_{2}})$. Since the homogeneous degree of the Rockland operator $\mathcal{R}_{2}^{-\frac{a_{2}}{\nu_{1}}}$ is  $0$, using \cite[Theorem 4.4.18]{FR16} we see again the independence of the spaces $L^{p}_{a_{1}, a_{2}}(\G)$ from the choice of Rockland operators $\R_{1}$ and $\R_{2}$.

Thus, we can omit the subscripts $\R$ or $\R_{1}$, $\R_{2}$ in the notation for these spaces. However, we may sometimes still write these to emphasise the particular norm that we use on these spaces.
\end{rem}

In \cite{RT-Heisenberg}, it was shown that if
$$
a>0,\; 1<r<\frac{Q}{a}\; \textrm{ and }\; 1\leq p\leq q\leq \frac{rQ}{Q-ar},
$$
then we have the following Gagliardo-Nirenberg type inequality
\begin{equation}\label{EQ:RT-GN}
\|u\|_{L^{q}(\G)}\lesssim \|u\|_{\dot{L}_{a}^{r}(\G)}^{s} \|u\|_{L^{p}(\G)}^{1-s}\simeq
 \|\R^{a/\nu}u\|_{L^{r}(\G)}^{s} \|u\|_{L^{p}(\G)}^{1-s},
\end{equation}
for $s=\left(\frac1p-\frac1q\right) \left(\frac{a}Q+\frac1p-\frac1r\right)^{-1}\in [0,1]$.

Such inequality was used in \cite{RT-Heisenberg} in the analysis of damped wave equations for Rockland operators on graded Lie groups, using the nonharmonic analysis pseudo-differential techniques
\cite{RT-IMRN}.
For completeness and also to fix the notation and the relation to Sobolev inequalities, we now give a simple proof of this result for the case $p=r$ relevant to our considerations.

\begin{thm}\label{THM-GN}
Let $\mathbb{G}$ be a graded Lie group of homogeneous dimension $Q$ and let $\mathcal{R}_{1}$ and $\mathcal{R}_{2}$ be positive Rockland operators of homogeneous degrees $\nu_{1}$ and $\nu_{2}$, respectively. Let $a_{1}> a_{2}\geq0$, $1<p<\frac{Q}{a_{1}}$ and $\frac{pQ}{Q-a_{2}p}\leq q\leq\frac{pQ}{Q-a_{1}p}$. Then there exists a positive constant $C$ such that
\begin{equation}\label{GN1}
\int_{\mathbb{G}}|u(x)|^{q}dx\leq C \left(\int_{\mathbb{G}}|\mathcal{R}_{1}^{\frac{a_{1}}{\nu_{1}}}u(x)|^{p}dx\right)^{\frac{Q(q-p)-a_{2}pq}{(a_{1}-a_{2})p^{2}}}
\left(\int_{\mathbb{G}}|\mathcal{R}_{2}^{\frac{a_{2}}{\nu_{2}}}u(x)|^{p}dx\right)^{\frac{a_{1}pq-Q(q-p)}{(a_{1}-a_{2})p^{2}}}
\end{equation}
holds for all $u\in \dot{L}^{p}_{a_{1}}(\mathbb{G})\cap \dot{L}^{p}_{a_{2}}(\mathbb{G})$.
\end{thm}

The proof of Theorem \ref{THM-GN} will be based on the following Sobolev embedding:
\begin{cor}[{\cite[Proposition 4.4.13, (5)]{FR16}}]
\label{cor_FR}
Let $\G$ be a graded Lie group of homogeneous dimension $Q$. Let $a>0$ and $1<p<q<\infty$ be such that
\begin{equation}\label{cor_ineq1}
Q\left(\frac{1}{p}-\frac{1}{q}\right)=a.
\end{equation}
Then we have
\begin{equation}\label{cor_ineq2}
\|u\|_{L^{q}(\G)}\lesssim \|u\|_{\dot{L}^{p}_a(\G)}\simeq \|\mathcal{R}^{\frac{a}{\nu}}u\|_{L^{p}(\G)},
\end{equation}
for all $u\in \dot{L}^{p}_{a}(\G)$,
where $\mathcal{R}$ is any positive Rockland operator on $\G$ of homogeneous degree $\nu$.
\end{cor}
We are now ready to prove Theorem \ref{THM-GN}.
\begin{proof}[Proof of Theorem \ref{THM-GN}] In the case $q=\frac{pQ}{Q-a_{2}p}$, we have $\frac{Q(q-p)-a_{2}pq}{(a_{1}-a_{2})p^{2}}=0$,  $\frac{a_{1}pq-Q(q-p)}{(a_{1}-a_{2})p^{2}}=\frac{q}{p}$ and $Q\left(\frac{1}{p}-\frac{1}{q}\right)=a_{2}$. Then, the inequality \eqref{GN1} has the form
$$\int_{\mathbb{G}}|u(x)|^{q}dx\leq C \left(\int_{\mathbb{G}}|\mathcal{R}_{2}^{\frac{a_{2}}{\nu_{2}}}u(x)|^{p}dx\right)^{\frac{q}{p}},$$
which is equivalent to the Sobolev inequality \eqref{cor_ineq2}.

In the case $q=\frac{pQ}{Q-a_{1}p}$, one gets $\frac{a_{1}pq-Q(q-p)}{(a_{1}-a_{2})p^{2}}=0$,  $\frac{Q(q-p)-a_{2}pq}{(a_{1}-a_{2})p^{2}}=\frac{q}{p}$ and $Q\left(\frac{1}{p}-\frac{1}{q}\right)=a_{1}$. Then, the inequality \eqref{GN1} has the form
$$\int_{\mathbb{G}}|u(x)|^{q}dx\leq C \left(\int_{\mathbb{G}}|\mathcal{R}_{1}^{\frac{a_{1}}{\nu_{1}}}u(x)|^{p}dx\right)^{\frac{q}{p}},$$
which is equivalent to the Sobolev inequality \eqref{cor_ineq2}.

Now let us consider the case $\frac{pQ}{Q-a_{2}p}<q<\frac{pQ}{Q-a_{1}p}$. We denote $p_{1}=\frac{pQ}{Q-a_{1}p}$ and $p_{2}=\frac{pQ}{Q-a_{2}p}$.    Using H\"{o}lder's inequality we obtain that
$$\int_{\mathbb{G}}|u(x)|^{q}dx=\int_{\mathbb{G}}|u(x)|^{qs}|u(x)|^{q(1-s)}dx\leq \left(\int_{\mathbb{G}}|u(x)|^{p_{1}}dx\right)^{\frac{qs}{p_{1}}}\left(\int_{\mathbb{G}}|u(x)|^{p_{2}}dx\right)^{\frac{q(1-s)}{p_{2}}},$$
where $\frac{qs}{p_{1}}+\frac{q(1-s)}{p_{2}}=1$. Then we find that $s=\frac{Q(q-p)-a_{2}pq}{(a_{1}-a_{2})pq}$, so that
$$\int_{\mathbb{G}}|u(x)|^{q}dx\leq\left(\int_{\mathbb{G}}|u(x)|^{p_{1}}dx\right)^{\frac{Q(q-p)-a_{2}pq}{(a_{1}-a_{2})pp_{1}}}
\left(\int_{\mathbb{G}}|u(x)|^{p_{2}}dx\right)^{\frac{a_{1}pq-Q(q-p)}{(a_{1}-a_{2})pp_{2}}}.$$
Then, using \eqref{cor_ineq2} for $\|u\|_{L^{p_1}(\G)}$ and $\|u\|_{L^{p_2}(\G)}$ we obtain \eqref{GN1}.
\end{proof}

\section{Existence of least energy solutions for  a class of nonlinear subelliptic equations}
\label{SEC:exist}

In this section we investigate the existence of least energy solutions to a class of nonlinear Schr\"{o}dinger type equations associated with the positive homogeneous Rockland operators.

So, let operators $\R_{1}$ and $\R_{2}$ be positive Rockland operators of homogeneous degrees $\nu_{1}$ and $\nu_{2}$, respectively, on a graded Lie group $\G$ of homogeneous dimension $Q$. Let $a_{1}> a_{2}\geq0$, $1<p<\frac{Q}{a_{1}}$ and $\frac{pQ}{Q-a_{2}p}<q<\frac{pQ}{Q-a_{1}p}$.
Let us consider the following Schr\"{o}dinger equation with the power nonlinearity
\begin{equation}\label{nonlinear}
\mathcal{R}_{1}^{\frac{a_{1}}{\nu_{1}}}(|\mathcal{R}_{1}^{\frac{a_{1}}{\nu_{1}}}u|^{p-2}\mathcal{R}_{1}^{\frac{a_{1}}{\nu_{1}}}u)+
\mathcal{R}_{2}^{\frac{a_{2}}{\nu_{2}}}(|\mathcal{R}_{2}^{\frac{a_{2}}{\nu_{2}}}u|^{p-2}\mathcal{R}_{2}^{\frac{a_{2}}{\nu_{2}}}u)=|u|^{q-2}u, \quad u\in L^{p}_{a_{1},a_{2}}(\mathbb{G}).
\end{equation}

Now, we briefly formulate some notations and definitions.

\begin{defn}
We say that the function $u\in L^{p}_{a_{1},a_{2}}(\mathbb{G})$ is a solution of \eqref{nonlinear} if and only if for any $\psi\in L^{p}_{a_{1},a_{2}}(\mathbb{G})$ the identity

$$\int_{\mathbb{G}}(|\mathcal{R}_{1}^{\frac{a_{1}}{\nu_{1}}}u(x)|^{p-2}\mathcal{R}_{1}^{\frac{a_{1}}{\nu_{1}}}
u(x)\overline{\mathcal{R}_{1}^{\frac{a_{1}}{\nu_{1}}}\psi(x)}+
\int_{\mathbb{G}}(|\mathcal{R}_{2}^{\frac{a_{2}}{\nu_{2}}}u(x)|^{p-2}\mathcal{R}_{2}^{\frac{a_{2}}{\nu_{2}}}
u(x)\overline{\mathcal{R}_{2}^{\frac{a_{2}}{\nu_{2}}}\psi(x)}$$
\begin{equation}\label{solution}
-|u(x)|^{q-2}u(x)\overline{\psi(x)})dx=0
\end{equation}
holds.
\end{defn}

By $\mathfrak{L}:L^{p}_{a_{1},a_{2}}(\mathbb{G})\to \mathbb R$ and $\mathfrak{I}:L^{p}_{a_{1},a_{2}}(\mathbb{G})\to \mathbb R$ we denote the following functionals acting on $L^{p}_{a_{1},a_{2}}(\mathbb{G})\cap L^{q}(\G)$:
\begin{equation}\label{L}
\mathfrak{L}(u):=\frac{1}{p}\int\limits_{\mathbb{G}}|\mathcal{R}_{1}^{\frac{a_{1}}{\nu_{1}}}u(x)|^{p}dx+
\frac{1}{p}\int\limits_{\mathbb{G}}|\mathcal{R}_{2}^{\frac{a_{2}}{\nu_{2}}}u(x)|^{p}dx
-\frac{1}{q}\int\limits_{\mathbb{G}}|u(x)|^{q}dx
\end{equation}
and
\begin{equation}\label{I}
\mathfrak{I}(u):=\int\limits_{\mathbb{G}}(|\mathcal{R}_{1}^{\frac{a_{1}}{\nu_{1}}}u(x)|^{p}+
|\mathcal{R}_{2}^{\frac{a_{2}}{\nu_{2}}}u(x)|^{p}-|u(x)|^{q})dx.
\end{equation}

We denote the Nehari set by
\begin{equation}\label{N}
\mathcal{N}:=\{u\in L^{p}_{a_{1},a_{2}}(\mathbb{G})\ \backslash\{0\}: \I(u)=0\}
\end{equation}
and we put
\begin{equation}\label{d}
d:=\inf\{\L(u):u\in\mathcal{N}\}.
\end{equation}

\begin{defn}\label{def}
Let $\Gamma$ be the set of the solutions of \eqref{nonlinear}, that is,
$$\Gamma=\{\phi\in L^{p}_{a_{1},a_{2}}(\mathbb{G}):\L'(\phi)=0\;\;{\rm and}\;\;\phi\neq0\}.$$
Let $\mathcal{G}$ be the set of least energy solutions of \eqref{nonlinear}, namely,
$$\mathcal{G}=\{u\in \Gamma: \L(u)\leq \L(\upsilon) \;\;{\rm for \;\;any}\;\;\upsilon\in\Gamma\}.$$
\end{defn}

Now we state the main result of this section.
\begin{thm}\label{thm1}
Let $a_{1}> a_{2}\geq0$, $1<p<\frac{Q}{a_{1}}$ and $\frac{pQ}{Q-a_{2}p}<q<\frac{pQ}{Q-a_{1}p}$. Then the Schr\"{o}dinger type equation \eqref{nonlinear} has a least energy solution $\phi\in L^{p}_{a_{1},a_{2}}(\mathbb{G})$.

Moreover, we have $d=\L(\phi)$.
\end{thm}

From now on in this section we assume that $a_{1},a_{2},p,q$ satisfy conditions of Theorem \ref{thm1}.
Before proving this theorem let us show the following lemmas.

\begin{lem}\label{LM: 2.1}
For all $u\in L^{p}_{a_{1},a_{2}}(\mathbb{G})\setminus\{0\},$ there exists a unique $\mu_{u}>0$ such that $\mu_{u}u\in\mathcal{N}$. Moreover, for $\I(u)<0$ we have $0<\mu_{u}<1$.
\end{lem}

\begin{proof}[Proof of Lemma \ref{LM: 2.1}]
It is easy to see that for an arbitrary $u\in L^{p}_{a_{1},a_{2}}(\mathbb{G})\setminus\{0\}$ and for
$$
\mu_{u}=\|u\|_{L^{p}_{a_{1}, a_{2}, \R_{1}, \R_{2}}(\mathbb{G})}^{\frac{p}{q-p}} \|u\|^{-\frac{q}{q-p}}_{L^{q}(\mathbb{G})},
$$
we get $\mu_{u}u\in\mathcal{N}$. Also, it is clear that $\mu_{u}$ is unique. Thus, from the
expression for $\mu_{u}$ we obtain that $0<\mu_{u}<1$ provided that  $\|u\|_{L^{p}_{a_{1}, a_{2}, \R_{1}, \R_{2}}(\mathbb{G})}^{p}<\|u\|^{q}_{L^{q}(\mathbb{G})}$.
\end{proof}

\begin{lem}\label{LM: 2.2}
We have $\underset{u\in\mathcal{N}}{\rm inf}\|u\|_{L^{p}_{a_{1}, a_{2}, \R_{1}, \R_{2}}(\mathbb{G})}>0$.
\end{lem}
\begin{proof}[Proof of Lemma \ref{LM: 2.2}]
Using the inequality \eqref{GN1}, we get for any $u\in\mathcal{N}$, that
\begin{equation*}
\begin{split}
\|u\|_{L^{p}_{a_{1}, a_{2}, \R_{1}, \R_{2}}(\mathbb{G})}^{p}=\|u\|^{q}_{L^{q}(\mathbb{G})}&\leq C \|\mathcal{R}_{1}^{\frac{a_{1}}{\nu_{1}}} u\|^{\frac{Q(q-p)-a_{2}pq}{(a_{1}-a_{2})p}}_{L^{p}(\mathbb{G})} \|\mathcal{R}_{2}^{\frac{a_{2}}{\nu_{2}}}u\|^{\frac{a_{1}pq-Q(q-p)}{(a_{1}-a_{2})p}}_{L^{p}(\mathbb{G})} \\
&\leq C \| u\|_{L^{p}_{a_{1}, a_{2}, \R_{1}, \R_{2}}(\mathbb{G})}^{\frac{Q(q-p)-a_{2}pq}{(a_{1}-a_{2})p}}
\|u\|_{L^{p}_{a_{1}, a_{2}, \R_{1}, \R_{2}}(\mathbb{G})}^{\frac{a_{1}pq-Q(q-p)}{(a_{1}-a_{2})p}}\\
&\leq C \| u\|_{L^{p}_{a_{1}, a_{2}, \R_{1}, \R_{2}}(\mathbb{G})}^{q}.
\end{split}
\end{equation*}
Thus, we obtain $\|u\|_{L^{p}_{a_{1}, a_{2}, \R_{1}, \R_{2}}(\mathbb{G})}^{q-p}\geq C^{-1}$. Putting $\kappa=C^{-\frac{1}{q-p}}>0$, we observe that $\|u\|_{L^{p}_{a_{1}, a_{2}, \R_{1}, \R_{2}}(\mathbb{G})}\geq \kappa$ for all $u\in\mathcal{N}$.
\end{proof}

\begin{lem}\label{LM: 2.3}
Assume that $a_{1}> a_{2}\geq0$, $1<p<\frac{Q}{a_{1}}$ and $\frac{pQ}{Q-a_{2}p}<q<\frac{pQ}{Q-a_{1}p}$. Let $D\subset\mathbb{G}$ be a smooth bounded domain. Then, we get the compact embedding $L^{p}_{a_{1}, a_{2}}(D)\hookrightarrow L^{q}(D)$.
\end{lem}
\begin{proof}[Proof of Lemma \ref{LM: 2.3}]
Let us take a bounded sequence $(u_{k})_{k\in\mathbb{N}}\subset L^{p}_{a_{1},a_{2}}(D)$. Suppose that the sequence $u_{k}$ converges to $u$ in the weak sense in $L^{p}_{a_{1}, a_{2}}(D)$. Then, by using the inequality \eqref{GN1}, we find
\begin{equation*}
\begin{split}
\|u_{k}-u\|^{q}_{L^{q}(\mathbb{G})}&\leq C \|\mathcal{R}_{1}^{\frac{a_{1}}{\nu_{1}}}u_{k}-\mathcal{R}_{1}^{\frac{a_{1}}{\nu_{1}}} u\|^{\frac{Q(q-p)-a_{2}pq}{(a_{1}-a_{2})p}}_{L^{p}(\mathbb{G})}
\|\mathcal{R}_{2}^{\frac{a_{2}}{\nu_{2}}}u_{k}-\mathcal{R}_{2}^{\frac{a_{2}}{\nu_{2}}}u\|^{\frac{a_{1}pq-Q(q-p)}{(a_{1}-a_{2})p}}_{L^{p}(\mathbb{G})}\\
&\leq C (\|u_{k}\|_{L^{p}_{a_{1}, a_{2}, \R_{1}, \R_{2}}(\mathbb{G})}+\|u\|_{L^{p}_{a_{1}, a_{2}, \R_{1}, \R_{2}}(\mathbb{G})})^{\frac{Q(q-p)-a_{2}pq}{(a_{1}-a_{2})p}}
\end{split}
\end{equation*}
$$\times\|\mathcal{R}_{2}^{\frac{a_{2}}{\nu_{2}}}u_{k}-\mathcal{R}_{2}^{\frac{a_{2}}
{\nu_{2}}}u\|^{\frac{a_{1}pq-Q(q-p)}{(a_{1}-a_{2})p}}_{L^{p}(\mathbb{G})}$$
$$\leq C (\|u_{k}\|_{L^{p}_{a_{1}, a_{2}, \R_{1}, \R_{2}}(\mathbb{G})}+\|u\|_{L^{p}_{a_{1}, a_{2}, \R_{1}, \R_{2}}(\mathbb{G})})^{\frac{Q(q-p)-a_{2}pq}{(a_{1}-a_{2})p}}$$
$$\times\|u_{k}-u\|^{\frac{a_{1}pq-Q(q-p)}{(a_{1}-a_{2})p}}_{L^{p}_{a_{1}, a_{2}, \R_{1}, \R_{2}}(\mathbb{G})}.$$
From the boundedness of the sequence
$(u_{k})_{k\in\mathbb{N}}$ in the space $L^{p}_{a_{1}, a_{2}}(D)$ we get $\|u_{k}-u\|_{L^{q}(\mathbb{G})}\to 0$ as $k\to\infty$. Thus, the lemma is proved.
\end{proof}

\begin{lem}\label{LM: 2.4}
If $v\in\N$ and $\L(v)=d$ then $v$ must be a least energy solution of the nonlinear equation \eqref{nonlinear}.
\end{lem}

\begin{proof}[Proof of Lemma \ref{LM: 2.4}]  One can conclude from the Lagrange
multiplier rule that there exists a real number $\theta$ such that for an arbitrary $\psi\in L^{p}_{a_{1},a_{2}}(\mathbb{G})$ we have
$$
\langle \L'(v), \psi \rangle_{\G}=\theta \langle \I'(v), \psi\rangle_{\G},
$$
due to the assumption on $v$. Here $\langle \cdot, \cdot \rangle_{\G}$ is a dual product between $L^{p}_{a_{1}, a_{2}}(\mathbb{G})$ and its dual space.

Since $q>p$, we get that
$$
\langle \I'(v), v\rangle_{\G} = p \| v \|_{L^{p}_{a_{1}, a_{2}, \R_{1}, \R_{2}}(\mathbb{G})}^{p} - q \int_{\mathbb{G}} |v|^{q} d x=(p-q) \int_{\mathbb{G}}|v|^{q} d x <0.
$$
From this we have
$$
\langle \L'(v), v\rangle_{\G}=\I(v)=0.
$$
Then due to facts
$$
\langle \I'(v), v\rangle_{\G}<0
$$
and
$$
\langle \L'(v), v\rangle_{\G}=0
$$
we obtain that $\theta = 0$. Thus, we get $\L'(v) = 0$. By taking into account Definition \ref{def}, we conclude that $v$ is a least energy solution of the nonlinear equation \eqref{nonlinear}.
Lemma \ref{LM: 2.4} is proved.
\end{proof}

Now we are in a position to prove Theorem \ref{thm1}.

\begin{proof}[Proof of Theorem \ref{thm1}] Choose $(v_{k})_{k}\subset\N$ as a minimising sequence. By using the Ekeland variational principle we know that there exists a sequence $(u_{k})_{k}\subset\N$ such that $\L(u_{k})\to d$ and $\L'(u_{k})\to 0$.
Then the Sobolev inequality and Lemma \ref{LM: 2.2} imply that there are two positive
constants $C_1$ and $C_2$ with the properties
$$
C_1\leq\|u_{k}\|_{L^{p}_{a_{1}, a_{2}, \R_{1}, \R_{2}}(\mathbb{G})}\leq C_2.
$$
Taking into account this and
$$
\|u_{k}\|_{L^{p}_{a_{1}, a_{2}, \R_{1}, \R_{2}}(\mathbb{G})}^{p}=\int_{\G}|u_{k}(x)|^{q}dx,
$$
one gets the existence of a positive constant
$C_3$ so that we have
\begin{equation}\label{EQ: (2.1)}
\limsup\limits_{k\to\infty}\int_{\G}|u_{k}(x)|^{q}dx\geq C_3 > 0.
\end{equation}
By Lemma \ref{LM: 2.3} and the concentration compactness argument of \cite[Lemma 3.1]{ST02}, we have that
 $u_{k}\to0$ in $L^{q}(\G)$ for all $\frac{pQ}{Q-a_{2}p}<q<\frac{pQ}{Q-a_{1}p}$
if it is true that
\begin{equation}
\lim\limits_{k\to\infty}\sup\limits_{\eta\in\G}\int\limits_{B(\eta, r)}|u_{k}(x)|^{q} d x = 0,
\end{equation}
for some $r > 0$, where $B(\eta, r)$ is a quasi-ball on $\G$ centred at $\eta$ with radius $r$. Then the fact that there is $C_4>0$ and $r > 1$ such that
\begin{equation}\label{EQ: (2.2)}
\liminf\limits_{k\to\infty}\sup\limits_{\eta\in\G}\int\limits_{B(\eta, r)}|u_{k}(x)|^{q}dx\geq C_4 > 0
\end{equation}
follows from the estimate \eqref{EQ: (2.1)}.
Assume that we have $\tilde{x}^{k}\in\G$ with
\begin{equation}\label{EQ: (2.3)}
\liminf\limits_{k\to\infty}\int\limits_{B(\tilde{x}^{k}, r)}|u_{k}(x)|^{q}dx\geq \frac{C_4}{2} > 0.
\end{equation}
By the left invariance of the Haar measure and of the operators $\R_{1}$, $\R_{2}$ we have
$$
\L(u_{k}(x\tilde{x}))=\L(u_{k}(x))$$
and
$$
\I(u_{k}(x\tilde{x}))=\I(u_{k}(x))
$$
for all $\tilde{x}^{k}\in\G$. Let us introduce $\omega_{k}(x):=u_{k}(x\tilde{x})$. Then we obtain $\L(\omega_{k})=\L(u_{k})$ and $\I(\omega_{k})=\I(u_{k})$. Moreover, it gives the bounded sequence $(\omega_{k})_{k}$ of elements of the space $L^{p}_{a_{1},a_{2}}(\mathbb{G})$ which satisfies
\begin{equation}\label{EQ: (2.4)}
\liminf\limits_{k\to\infty}\int_{B(0, r)}|\omega_{k}(x)|^{q}dx\geq \frac{C_4}{2} > 0.
\end{equation}
There is a subsequence, denoted by $\omega_{k}$ that weakly converges to $\phi$ in the space $L^{p}_{a_{1},a_{2}}(\mathbb{G})$. Then Lemma \ref{LM: 2.3} implies that $\omega_{k}$ strongly converges to $\phi$ in $L^{q}_{loc}(\G)$. Due to this and the estimate \eqref{EQ: (2.4)}, finally, we obtain that $\phi\neq0$.

Now we are in a position to show that $\omega_{k}$ converges strongly to $\phi$ in the space $L^{p}_{a_{1},a_{2}}(\mathbb{G})$. We will show first that $\I(\phi)=0$. Suppose that $\I(\phi)<0$. Lemma \ref{LM: 2.1} implies that there exists a positive number $\mu_{\phi} < 1$ such that $\mu_{\phi}\phi\in\N$ for $\I(\phi) < 0$. Since $\I(\omega_{k}) = 0$, by the Fatou lemma we get
\begin{equation}\label{EQ: (2.5)}
\begin{split}
d+o(1)=\L(\omega_{k}) & =\left(\frac{1}{p}-\frac{1}{q}\right)\int_{\G} |\omega_{k}(x)|^{q}dx\\
& \geq \left(\frac{1}{p}-\frac{1}{q}\right)\int_{\G} |\phi(x)|^{q}dx + o(1) \\
& = \left(\frac{1}{p}-\frac{1}{q}\right)\mu_{\phi}^{-q}\int_{\G} |\mu_{\phi}\phi(x)|^{q}dx + o(1)\\
& =\mu_{\phi}^{-q} \L(\mu_{\phi}\phi)  + o(1).
\end{split}
\end{equation}
Then, from the property $0 < \mu_{\phi} < 1$ we obtain that $d > \L(\mu_{\phi}\phi)$. Since $\mu_{\phi}\phi\in\N$, we get a contradiction.

Suppose now that $\I(\phi)>0$. We use the following lemma:
\begin{lem}[{\cite[Lemma 3]{BL83}}]
\label{BrL_lem}
 Let $\ell:\mathbb{C}\rightarrow\mathbb{R}$ be convex and let $m>1$. Then
$$|\ell(a+b)-\ell(a)|\leq \varepsilon[\ell(ma)-m\ell(a)]+|\ell(C_{\varepsilon}b)|+|\ell(-C_{\varepsilon}b)|$$
for all $a,b \in \mathbb{C}$, $0<\varepsilon<\frac{1}{m}$ and $\frac{1}{C_{\varepsilon}}=\varepsilon(m-1)$.
\end{lem}
Applying this lemma for $\psi_{k}=\omega_{k}-\phi$ we get
$$
0=\I(\omega_{k})=\I(\phi)+\I(\psi_{k})+o(1)
$$
in the case $\I(\phi) > 0$. Consequently, we have
\begin{equation}\label{EQ: (2.6)}
\limsup\limits_{k\to\infty} \I(\psi_{k})< 0.
\end{equation}
Using Lemma \ref{LM: 2.1}, there is a positive constant $\mu_{k}:=\mu_{\psi_{k}}$ with the properties $\mu_{k} \psi_{k}\in \mathcal{N}$ and $\limsup\limits_{k\to\infty}\mu_{k}\in(0, 1)$. Indeed, suppose that $\limsup\limits_{k\to\infty}\mu_{k}=1$. Then we get a subsequence $(\mu_{k_{j}})_{j\in\mathbb{N}}$ such that $\lim\limits_{j\to\infty}\mu_{k_{j}}=1$. From the property $\mu_{k_{j}}\psi_{k_{j}}\in \mathcal{N}$ we obtain $\I(\psi_{k_{j}})=\I(\mu_{k_{j}}\psi_{k_{j}})+o(1)=o(1)$. But this is impossible in view of \eqref{EQ: (2.6)}. Thus, we have that $\limsup\limits_{k\to\infty}\mu_{k}\in(0,1)$. We now observe that
\begin{equation}\label{EQ: (2.7)}
\begin{split}
d+o(1)&=\L(\omega_{k})=\left(\frac{1}{p}-\frac{1}{q}\right)\int_{\mathbb{G}}|\omega_{k}(x)|^{q}dx\\
&\geq\left(\frac{1}{p}-\frac{1}{q}\right)\int_{\mathbb{G}}|\psi_{k}(x)|^{q}dx\\
&=\left(\frac{1}{p}-\frac{1}{q}\right)\mu_{k}^{-q}\int_{\mathbb{G}}|\mu_{k}\psi_{k}(x)|^{q}dx+o(1)\\
&=\mu_{k}^{-q}\L(\mu_{k}\psi_{k})+o(1).
\end{split}
\end{equation}
The fact $\limsup\limits_{k\to\infty}\mu_{k}\in(0,1)$ implies that $d>\L(\mu_{k}\psi_{k})$. It means that we obtain a contradiction since $\mu_{k}\phi_{k}\in\N$.

Thus, we must have $\I(\phi)=0$, so that also $\L(\phi)\geq d$. Now let us prove that $\psi_{k}=\omega_{k}-\phi\to0$ in $L^{p}_{a_{1},a_{2}}(\mathbb{G})$. Indeed, if it is not true, that is, $\|\psi_{k}\|_{L^{p}_{a_{1},a_{2}, \R_{1}, \R_{2}}(\mathbb{G})}$ does not converge to zero as $k\to\infty$, then we have the following cases. The first case is when $\int_{\G}|\psi_{k}(x)|^{q}dx$ does not vanish as $k\to\infty$. Then the equalities
$$
0=\I(\omega_{k})=\I(\phi)+\I(\psi_{k})+o(1)=\I(\psi_{k})+o(1)
$$
and the Brezis-Lieb lemma give us the following contradiction
\begin{equation*}
\begin{split}
d+o(1)&=\L(\omega_{k})=\L(\phi)+\L(\psi_{k})+o(1)\\
&\geq d+d+o(1).
\end{split}
\end{equation*}
The second case is when $\int_{\G}|\psi_{k}(x)|^{q}dx$ vanishes as $k\to\infty$. In this case, we observe also a contradiction:
\begin{equation*}
\begin{split}
d+o(1)&=\L(\omega_{k})=\L(\phi)+\frac{1}{p}\|\psi_{k}\|_{L^{p}_{a_{1}, a_{2}, \R_{1}, \R_{2}}(\mathbb{G})}^{p}+o(1)\\
&\geq d+\frac{1}{p}\|\psi_{k}\|_{L^{p}_{a_{1}, a_{2}, \R_{1}, \R_{2}}(\mathbb{G})}^{p}+o(1)>d,
\end{split}
\end{equation*}
if $\|\psi_{k}\|_{L^{p}_{a_{1}, a_{2}, \R_{1}, \R_{2}}(\mathbb{G})}\nrightarrow0$ as $k\to\infty$. Finally, we obtain that $\omega_{k}$ converges strongly to $\phi$ in the Sobolev space $L^{p}_{a_{1},a_{2}}(\mathbb{G})$ and that $\phi$ is a minimiser of $d$. Lemma \ref{LM: 2.4} guarantees the fact that $\phi$ is a least energy solution of \eqref{nonlinear}.
Theorem \ref{thm1} is proved.
\end{proof}

\section{Best constants in the Gagliardo-Nirenberg inequalities}
\label{SEC:sharp}

In this section we obtain a sharp expression for the smallest positive constant $C$ in \eqref{GN1}.
We denote by $C_{GN, \R_{1},\R_{2}}=C_{GN, \R_{1},\R_{2}, a_{1},a_{2},p,q}$ the smallest positive constant $C$ such that the Gagliardo-Nirenberg inequality \eqref{GN1} holds.
Now let us show the main result of this section.

\begin{thm}\label{sharp}
Let $a_{1}> a_{2}\geq0$, $1<p<\frac{Q}{a_{1}}$ and $\frac{pQ}{Q-a_{2}p}<q<\frac{pQ}{Q-a_{1}p}$. Let $\phi$ be a least energy solution of \eqref{nonlinear} and let $d$ be defined in \eqref{d}. Let $C_{GN, \R_{1}, \R_{2}}$ be the smallest positive constant $C$ in \eqref{GN1}. Then we have
$$C_{GN, \R_{1}, \R_{2}}=\frac{(a_{1}-a_{2})pq}{a_{1}pq-Q(q-p)}
\left(\frac{Q(q-p)-a_{2}pq}{a_{1}pq-Q(q-p)}\right)^{\frac{a_{2}pq-Q(q-p)}{(a_{1}-a_{2})p^{2}}}
\|\mathcal{R}_{2}^{\frac{a_{2}}{\nu_{2}}}\phi\|_{L^{p}(\mathbb{G})}^{p-q}$$
\begin{equation}\label{sharp1} =
\frac{(a_{1}-a_{2})pq}{a_{1}pq-Q(q-p)}
\left(\frac{Q(q-p)-a_{2}pq}{a_{1}pq-Q(q-p)}\right)^{\frac{a_{2}pq-Q(q-p)}{(a_{1}-a_{2})p^{2}}}
\left(\frac{a_{1}pq-Q(q-p)}{(a_{1}-a_{2})(q-p)}d\right)^{\frac{p-q}{p}}.
\end{equation}
\end{thm}
The proof of Theorem \ref{sharp} will be based on the following lemmas:

\begin{lem}\label{sharp_lem}
Let $\phi$ be a least energy solution of \eqref{nonlinear}. Then we have
\begin{equation}\label{sharp_lem1}\int_{\mathbb{G}}|\mathcal{R}_{1}^{\frac{a_{1}}{\nu_{1}}}\phi(x)|^{p}dx=
\frac{Q(q-p)-a_{2}pq}{a_{1}pq-Q(q-p)}\int_{\mathbb{G}}|\mathcal{R}_{2}^{\frac{a_{2}}{\nu_{2}}}\phi(x)|^{p}dx
\end{equation}
and
\begin{equation}\label{sharp_lem2}\int_{\mathbb{G}}|\phi(x)|^{q}dx=\frac{(a_{1}-a_{2})pq}{a_{1}pq-Q(q-p)}
\int_{\mathbb{G}}|\mathcal{R}_{2}^{\frac{a_{2}}{\nu_{2}}}\phi(x)|^{p}dx.\end{equation}
We also have
\begin{equation}\label{sharp_d}
\int_{\mathbb{G}}|\mathcal{R}_{2}^{\frac{a_{2}}{\nu_{2}}}\phi(x)|^{p}dx=\frac{a_{1}pq-Q(q-p)}{(a_{1}-a_{2})(q-p)}d.
\end{equation}
\end{lem}
\begin{proof}[Proof of Lemma \ref{sharp_lem}] Using the fact that $\phi$ be a least energy solution of \eqref{nonlinear}, we obtain from \eqref{solution} that
\begin{equation}\label{sharp_lem3}
\int_{\mathbb{G}}|\mathcal{R}_{1}^{\frac{a_{1}}{\nu_{1}}}\phi(x)|^{p}dx+\int_{\mathbb{G}}|\mathcal{R}_{2}^{\frac{a_{2}}{\nu_{2}}}\phi(x)|^{p}dx
=\int_{\mathbb{G}}|\phi(x)|^{q}dx.
\end{equation}
On the other hand, for $\lambda>0$ and $\widetilde{\phi}_{\lambda}(x):=\lambda^{\frac{Q}{p}}\phi(\delta_{\lambda}(x))$ we have the relation
$$\L(\widetilde{\phi}_{\lambda}(x))=\frac{\lambda^{Q}}{p}\int_{\G}|\R_{1}^{\frac{a_{1}}{\nu_{1}}}\phi(\delta_{\lambda}(x))|^{p}dx
+\frac{\lambda^{Q}}{p}\int_{\G}|\R_{2}^{\frac{a_{2}}{\nu_{2}}}\phi(\delta_{\lambda}(x))|^{p}dx
-\frac{\lambda^{\frac{Qq}{p}}}{q}\int_{\G}|\phi(\delta_{\lambda}(x))|^{q}dx$$
$$=\frac{\lambda^{a_{1}p}}{p}\int_{\G}|\R_{1}^{\frac{a_{1}}{\nu_{1}}}\phi(x)|^{p}dx
+\frac{\lambda^{a_{2}p}}{p}\int_{\G}|\R_{2}^{\frac{a_{2}}{\nu_{2}}}\phi(x)|^{p}dx
-\frac{\lambda^{\frac{Qq}{p}-Q}}{q}\int_{\G}|\phi(x)|^{q}dx.$$
It follows that
\begin{equation}\label{sharp_lem4} 0=\frac{\partial}{\partial \lambda}\L(\widetilde{\phi}_{\lambda})|_{\lambda=1}=a_{1}\int_{\mathbb{G}}|\mathcal{R}_{1}^{\frac{a_{1}}{\nu_{1}}}\phi(x)|^{p}dx
+a_{2}\int_{\mathbb{G}}|\mathcal{R}_{2}^{\frac{a_{2}}{\nu_{2}}}\phi(x)|^{p}dx
-\frac{Q(q-p)}{pq}\int_{\mathbb{G}}|\phi(x)|^{q}dx.
\end{equation}
The equalities \eqref{sharp_lem3} and \eqref{sharp_lem4} imply \eqref{sharp_lem1} and \eqref{sharp_lem2}.

To show \eqref{sharp_d}, taking into account \eqref{norm} and \eqref{L}, we have by \eqref{sharp_lem2},
\begin{multline*}
d=\L(\phi)=\frac{1}{p}\|\phi\|_{L^{p}_{a_{1}, a_{2}, \R_{1}, \R_{2}}(\mathbb{G})}^{p}-\frac{1}{q}
\|\phi\|^{q}_{L^{q}(\mathbb{G})}\\
=\left(\frac{1}{p}-\frac{1}{q}\right)\|\phi\|^{q}_
{L^{q}(\mathbb{G})}=\frac{(a_{1}-a_{2})(q-p)}{a_{1}pq-Q(q-p)}
\int_{\mathbb{G}}|\mathcal{R}_{2}^{\frac{a_{2}}{\nu_{2}}}\phi(x)|^{p}dx,
\end{multline*}
so that it follows that
\begin{equation*}
\int_{\mathbb{G}}|\mathcal{R}_{2}^{\frac{a_{2}}{\nu_{2}}}\phi(x)|^{p}dx=\frac{a_{1}pq-Q(q-p)}{(a_{1}-a_{2})(q-p)}d,
\end{equation*}
which is \eqref{sharp_d}.
\end{proof}

\begin{lem}\label{sharp_lem2} Let
\begin{equation}\label{T_rho}
T_{\rho, p, q}:=\inf\left\{\|u\|_{L^{p}_{a_{1}, a_{2}, \R_{1}, \R_{2}}(\mathbb{G})}^{p}:u\in L^{p}_{a_{1},a_{2}}(\mathbb{G}) \;\;{\rm and}\;\; \int_{\mathbb{G}}|u(x)|^{q}dx=\rho\right\},
\end{equation}
for $\rho>0$. If $\phi$ is a minimiser obtained in Theorem \ref{thm1}, then $\phi$ is a minimiser of $T_{\rho_{0}, p, q}$ with $\rho_{0}=\int_{\mathbb{G}}|\phi(x)|^{q}dx$.
\end{lem}
\begin{proof}[Proof of Lemma \ref{sharp_lem2}]
By definition of $T_{\rho_{0}, p, q}$ we can note that $\|\phi\|_{L^{p}_{a_{1}, a_{2}, \R_{1}, \R_{2}}(\mathbb{G})}^{p}\geq T_{\rho_{0}, p, q}$. Now let us show that $T_{\rho_{0}, p, q}\geq\|\phi\|_{L^{p}_{a_{1}, a_{2}, \R_{1}, \R_{2}}(\mathbb{G})}^{p}$. For all $u\in L^{p}_{a_{1}, a_{2}}(\mathbb{G})$ satisfying $\int_{\mathbb{G}}|u(x)|^{q}dx=\int_{\mathbb{G}}|\phi(x)|^{q}dx$, by Lemma \ref{LM: 2.1} there exists a unique
$$\lambda_{0}=\|u\|_{L^{p}_{a_{1}, a_{2}, \R_{1}, \R_{2}}(\mathbb{G})}^{\frac{p}{q-p}}\left(\int_{\G}|u|^{q}dx\right)^{-\frac{1}{q-p}}$$
such that $\I(\lambda_{0}u)=0$. Using the fact that $\lambda_{0}u\neq0$ and $\phi$ achieves the minimum $d$, by a direct calculation one has
$$\left(\frac{1}{p}-\frac{1}{q}\right)\|\phi\|_{L^{p}_{a_{1}, a_{2}, \R_{1}, \R_{2}}(\mathbb{G})}^{p}=\L(\phi)\leq \L(\lambda_{0}u)=\left(\frac{1}{p}-\frac{1}{q}\right)\lambda_{0}^{p}\|u\|_{L^{p}_{a_{1}, a_{2}, \R_{1}, \R_{2}}(\mathbb{G})}^{p}$$
$$=\left(\frac{1}{p}-\frac{1}{q}\right)\|u\|_{L^{p}_{a_{1}, a_{2}, \R_{1}, \R_{2}}(\mathbb{G})}
^{\frac{p^{2}}{q-p}}\left(\int_{\mathbb{G}}|u(x)|^{q}dx\right)^{-\frac{p}{q-p}}\|u\|_{L^{p}_{a_{1}, a_{2}, \R_{1}, \R_{2}}(\mathbb{G})}^{p}.$$
Then, since $\int_{\mathbb{G}}|u(x)|^{q}dx=\int_{\mathbb{G}}|\phi(x)|^{q}dx$ and $\int_{\mathbb{G}}|\phi(x)|^{q}dx=\|\phi\|_{L^{p}_{a_{1}, a_{2}, \R_{1}, \R_{2}}(\mathbb{G})}^{p}$, we get that $\|\phi\|_{L^{p}_{a_{1}, a_{2}, \R_{1}, \R_{2}}(\mathbb{G})}^{p}=\|u\|_{L^{p}_{a_{1}, a_{2}, \R_{1}, \R_{2}}(\mathbb{G})}^{p}$. From the arbitrariness of $u$ we obtain $T_{\rho_{0}, p, q}\geq\|\phi\|_{L^{p}_{a_{1}, a_{2}, \R_{1}, \R_{2}}(\mathbb{G})}^{p}$. Thus, $T_{\rho_{0}, p, q}=\|\phi\|_{L^{p}_{a_{1}, a_{2}, \R_{1}, \R_{2}}(\mathbb{G})}^{p}$, which gives the fact that $\phi$ is a minimiser of $T_{\rho_{0}, p, q}$.
\end{proof}
We are now ready to prove Theorem \ref{sharp}.
\begin{proof}[Proof of Theorem \ref{sharp}] Let us denote
\begin{equation}\label{J}
J(u):=\left(\int_{\mathbb{G}}|\mathcal{R}_{1}^{\frac{a_{1}}{\nu_{1}}}u(x)|^{p}dx\right)^{\frac{Q(q-p)-a_{2}pq}{(a_{1}-a_{2})p^{2}}}
\left(\int_{\mathbb{G}}|\mathcal{R}_{2}^{\frac{a_{2}}{\nu_{2}}}u(x)|^{p}dx\right)^
{\frac{a_{1}pq-Q(q-p)}{(a_{1}-a_{2})p^{2}}}\left(\int_{\mathbb{G}}|u(x)|^{q}dx\right)^{-1}.
\end{equation}
Then the sharp expression $C_{GN, \R_{1}, \R_{2}}$ can be estimated by studying the following minimisation problem
$$C_{GN, \R_{1}, \R_{2}}^{-1}=\inf\{J(u):u\in L^{p}_{a_{1},a_{2}}(\mathbb{G})\ \backslash\{0\}\}.$$
Since $\phi(x)\neq0$ and $\phi \in  L^{p}_{a_{1},a_{2}}(\mathbb{G})$, using  Lemma \ref{sharp_lem} we calculate
$$J(\phi)=\left(\int_{\mathbb{G}}|\mathcal{R}_{1}^{\frac{a_{1}}{\nu_{1}}}\phi(x)|^{p}dx\right)
^{\frac{Q(q-p)-a_{2}pq}{(a_{1}-a_{2})p^{2}}}
\left(\int_{\mathbb{G}}|\mathcal{R}_{2}^{\frac{a_{2}}{\nu_{2}}}\phi(x)|^{p}dx\right)^
{\frac{a_{1}pq-Q(q-p)}{(a_{1}-a_{2})p^{2}}}\left(\int_{\mathbb{G}}|\phi(x)|^{q}dx\right)^{-1}$$
$$=\left(\frac{Q(q-p)-a_{2}pq}{a_{1}pq-Q(q-p)}\int_{\mathbb{G}}
|\mathcal{R}_{2}^{\frac{a_{2}}{\nu_{2}}}\phi(x)|^{p}dx\right)
^{\frac{Q(q-p)-a_{2}pq}{(a_{1}-a_{2})p^{2}}}
\left(\int_{\mathbb{G}}|\mathcal{R}_{2}^{\frac{a_{2}}{\nu_{2}}}\phi(x)|^{p}dx\right)^
{\frac{a_{1}pq-Q(q-p)}{(a_{1}-a_{2})p^{2}}}$$
$$\times\left(\frac{(a_{1}-a_{2})pq}{a_{1}pq-Q(q-p)}\int_{\mathbb{G}}|\mathcal{R}_{2}^{\frac{a_{2}}{\nu_{2}}}
\phi(x)|^{p}dx\right)^{-1}$$
$$=\frac{a_{1}pq-Q(q-p)}{(a_{1}-a_{2})pq}\left(\frac{Q(q-p)-a_{2}pq}{a_{1}pq-Q(q-p)}\right)
^{\frac{Q(q-p)-a_{2}pq}{(a_{1}-a_{2})p^{2}}}
\left(\int_{\mathbb{G}}|\mathcal{R}_{2}^{\frac{a_{2}}{\nu_{2}}}\phi(x)|^{p}dx\right)^
{\frac{q-p}{p}}.$$
It follows that
\begin{equation}\label{sharp_thm001}
C_{GN, \R_{1}, \R_{2}}^{-1}\leq \frac{a_{1}pq-Q(q-p)}{(a_{1}-a_{2})pq}\left(\frac{Q(q-p)-a_{2}pq}{a_{1}pq-Q(q-p)}\right)
^{\frac{Q(q-p)-a_{2}pq}{(a_{1}-a_{2})p^{2}}}
\left(\int_{\mathbb{G}}|\mathcal{R}_{2}^{\frac{a_{2}}{\nu_{2}}}\phi(x)|^{p}dx\right)^
{\frac{q-p}{p}}.
\end{equation}
Now let us obtain a lower estimate for the constant $C_{GN, \R_{1}, \R_{2}}^{-1}$. For positive parameters $\lambda$ and $\mu$, and for any $u\in L^{p}_{a_{1},a_{2}}(\mathbb{G})\ \backslash\{0\}$, we define $\omega(x):=\lambda u(\delta_{\mu}(x))$. Then, a direct calculation gives that
\begin{equation}\label{sharp_thm01}
\int_{\mathbb{G}}|\mathcal{R}_{1}^{\frac{a_{1}}{\nu_{1}}}\omega(x)|^{p}dx=\lambda^{p}\mu^{a_{1}p-Q}\int_{\mathbb{G}}
|\mathcal{R}_{1}^{\frac{a_{1}}{\nu_{1}}}u(x)|^{p}dx,
\end{equation}
$$\int_{\mathbb{G}}|\mathcal{R}_{2}^{\frac{a_{2}}{\nu_{2}}}\omega(x)|^{p}dx=\lambda^{p}\mu^{a_{2}p-Q}\int_{\mathbb{G}}
|\mathcal{R}_{2}^{\frac{a_{2}}{\nu_{2}}}u(x)|^{p}dx$$
and
$$\int_{\mathbb{G}}|\omega(x)|^{q}dx=\lambda^{q}\mu^{-Q}\int_{\mathbb{G}}|u(x)|^{q}dx.$$
We choose $\lambda$ and $\mu$ such that
\begin{equation}\label{sharp_thm1}\lambda^{p}\mu^{a_{2}p-Q}\int_{\mathbb{G}}
|\mathcal{R}_{2}^{\frac{a_{2}}{\nu_{2}}}u(x)|^{p}dx=\int_{\mathbb{G}}|\mathcal{R}_{2}^{\frac{a_{2}}{\nu_{2}}}\phi(x)|^{p}dx
\end{equation}
and
\begin{equation}\label{sharp_thm2}\lambda^{q}\mu^{-Q}\int_{\mathbb{G}}|u(x)|^{q}dx=\int_{\mathbb{G}}|\phi(x)|^{q}dx.\end{equation}
By \eqref{sharp_thm1}, \eqref{sharp_thm2} and \eqref{sharp_lem2}, one has
$$\lambda^{q}\mu^{-Q}\int_{\mathbb{G}}|u(x)|^{q}dx=\int_{\mathbb{G}}|\phi(x)|^{q}dx=\frac{(a_{1}-a_{2})pq}
{a_{1}pq-Q(q-p)}\int_{\mathbb{G}}|\mathcal{R}_{2}^{\frac{a_{2}}{\nu_{2}}}\phi(x)|^{p}dx$$
$$=\frac{(a_{1}-a_{2})pq}
{a_{1}pq-Q(q-p)}\lambda^{p}\mu^{a_{2}p-Q}\int_{\mathbb{G}}|\mathcal{R}_{2}^{\frac{a_{2}}{\nu_{2}}}u(x)|^{p}dx.$$
It follows that
$$\lambda^{p}=\left(\frac{(a_{1}-a_{2})pq}
{a_{1}pq-Q(q-p)}\right)^{\frac{p}{q-p}}\mu^{\frac{a_{2}p^{2}}{q-p}}
\left(\int_{\mathbb{G}}|\mathcal{R}_{2}^{\frac{a_{2}}{\nu_{2}}}u(x)|^{p}dx\right)^{\frac{p}{q-p}}
\left(\int_{\mathbb{G}}|u(x)|^{q}dx\right)^{-\frac{p}{q-p}}.$$
From this and \eqref{sharp_thm1}, we have
$$\mu^{a_{1}p-Q}=\left(\frac{(a_{1}-a_{2})pq}
{a_{1}pq-Q(q-p)}\right)^{\frac{p}{q-p}\cdot
\frac{(a_{1}p-Q)(q-p)}{Q(q-p)-a_{2}pq}}\left(\int_{\mathbb{G}}|\mathcal{R}_{2}^{\frac{a_{2}}{\nu_{2}}}
u(x)|^{p}dx\right)^{\frac{q}{q-p}\cdot\frac{(a_{1}p-Q)(q-p)}{Q(q-p)-a_{2}pq}}$$
$$\times\left(\int_{\mathbb{G}}|u(x)|^{q}dx\right)^{-\frac{p}{q-p}\cdot\frac{(a_{1}p-Q)(q-p)}{Q(q-p)-a_{2}pq}}
\left(\int_{\mathbb{G}}|\mathcal{R}_{2}^{\frac{a_{2}}{\nu_{2}}}\phi(x)|^{p}dx\right)^{-\frac{(a_{1}p-Q)(q-p)}{Q(q-p)-a_{2}pq}}$$
$$=\left(\frac{(a_{1}-a_{2})pq}
{a_{1}pq-Q(q-p)}\right)^{\frac{p(a_{1}p-Q)}{Q(q-p)-a_{2}pq}}\left(\int_{\mathbb{G}}|\mathcal{R}_{2}^{\frac{a_{2}}{\nu_{2}}}
u(x)|^{p}dx\right)^{\frac{q(a_{1}p-Q)}{Q(q-p)-a_{2}pq}}$$
\begin{equation}\label{sharp_thm4}\times\left(\int_{\mathbb{G}}|u(x)|^{q}dx\right)^{-\frac{p(a_{1}p-Q)}{Q(q-p)-a_{2}pq}}
\left(\int_{\mathbb{G}}|\mathcal{R}_{2}^{\frac{a_{2}}{\nu_{2}}}\phi(x)|^{p}dx\right)^{-\frac{(a_{1}p-Q)(q-p)}{Q(q-p)-a_{2}pq}},\end{equation}
and
$$\lambda^{p}=\left(\frac{(a_{1}-a_{2})pq}
{a_{1}pq-Q(q-p)}\right)^{\frac{p(Q-a_{2}p)}{Q(q-p)-a_{2}pq}}\left(\int_{\mathbb{G}}|\mathcal{R}_{2}^{\frac{a_{2}}{\nu_{2}}}
u(x)|^{p}dx\right)^{\frac{q(Q-a_{2}p)}{Q(q-p)-a_{2}pq}}$$
$$\times\left(\int_{\mathbb{G}}|u(x)|^{q}dx\right)^{-\frac{p(Q-a_{2}p)}{Q(q-p)-a_{2}pq}}
\left(\int_{\mathbb{G}}|\mathcal{R}_{2}^{\frac{a_{2}}{\nu_{2}}}\phi(x)|^{p}dx\right)^{-\frac{(Q-a_{2}p)(q-p)}{Q(q-p)-a_{2}pq}}
$$
\begin{equation}\label{sharp_thm3}\times\left(\int_{\mathbb{G}}|\mathcal{R}_{2}^{\frac{a_{2}}{\nu_{2}}}u(x)|^{p}dx\right)^{-1}
\int_{\mathbb{G}}|\mathcal{R}_{2}^{\frac{a_{2}}{\nu_{2}}}\phi(x)|^{p}dx.
\end{equation}
By \eqref{sharp_thm1} and \eqref{sharp_thm2} we have
$$
\int_{\mathbb{G}}|\mathcal{R}_{2}^{\frac{a_{2}}{\nu_{2}}}\omega(x)|^{p}dx=\int_{\mathbb{G}}|\mathcal{R}_{2}^{\frac{a_{2}}{\nu_{2}}}\phi(x)|^{p}dx \;\textrm{ and }\;\int_{\mathbb{G}}|\omega(x)|^{q}dx=\int_{\mathbb{G}}|\phi(x)|^{q}dx.
$$
Since $\phi$ is a minimiser of $T_{\rho_{0}, p, q}$ with $\rho_{0}=\int_{\mathbb{G}}|\phi(x)|^{q}dx$ by Lemma \ref{sharp_lem2}, we obtain that
$$\int_{\mathbb{G}}|\mathcal{R}_{1}^{\frac{a_{1}}{\nu_{1}}}\omega(x)|^{p}dx\geq
\int_{\mathbb{G}}|\mathcal{R}_{1}^{\frac{a_{1}}{\nu_{1}}}\phi(x)|^{p}dx,$$
and using \eqref{sharp_thm01} and \eqref{sharp_lem1}, one gets
$$\lambda^{p}\mu^{a_{1}p-Q}\int_{\mathbb{G}}
|\mathcal{R}_{1}^{\frac{a_{1}}{\nu_{1}}}u(x)|^{p}dx=\int_{\mathbb{G}}|\mathcal{R}_{1}^{\frac{a_{1}}{\nu_{1}}}\omega(x)|^{p}dx$$
$$\geq\int_{\mathbb{G}}|\mathcal{R}_{1}^{\frac{a_{1}}{\nu_{1}}}\phi(x)|^{p}dx=
\frac{Q(q-p)-a_{2}pq}{a_{1}pq-Q(q-p)}\int_{\mathbb{G}}|\mathcal{R}_{2}^{\frac{a_{2}}{\nu_{2}}}\phi(x)|^{p}dx.$$
Putting here \eqref{sharp_thm3} and \eqref{sharp_thm4}, one has
$$\left(\frac{(a_{1}-a_{2})pq}
{a_{1}pq-Q(q-p)}\right)^{\frac{(a_{1}-a_{2})p^{2}}{Q(q-p)-a_{2}pq}}
\left(\int_{\mathbb{G}}|\mathcal{R}_{2}^{\frac{a_{2}}{\nu_{2}}}
u(x)|^{p}dx\right)^{\frac{(a_{1}-a_{2})pq}{Q(q-p)-a_{2}pq}}$$
$$\times\left(\int_{\mathbb{G}}|u(x)|^{q}dx\right)^{-\frac{(a_{1}-a_{2})p^{2}}{Q(q-p)-a_{2}pq}}
\left(\int_{\mathbb{G}}|\mathcal{R}_{2}^{\frac{a_{2}}{\nu_{2}}}\phi(x)|^{p}dx\right)^{-\frac{(a_{1}-a_{2})p(q-p)}{Q(q-p)-a_{2}pq}}
$$
$$\times\left(\int_{\mathbb{G}}|\mathcal{R}_{2}^{\frac{a_{2}}{\nu_{2}}}u(x)|^{p}dx\right)^{-1}
\int_{\mathbb{G}}|\mathcal{R}_{2}^{\frac{a_{2}}{\nu_{2}}}\phi(x)|^{p}dx$$
$$\times\int_{\mathbb{G}}
|\mathcal{R}_{1}^{\frac{a_{1}}{\nu_{1}}}u(x)|^{p}dx\geq \frac{Q(q-p)-a_{2}pq}{a_{1}pq-Q(q-p)}\int_{\mathbb{G}}|\mathcal{R}_{2}^{\frac{a_{2}}{\nu_{2}}}\phi(x)|^{p}dx.$$
Taking into account the definition of $J(u)$ in \eqref{J}, we obtain
$$J(u)\geq \frac{a_{1}pq-Q(q-p)}{(a_{1}-a_{2})pq}\left(\frac{Q(q-p)-a_{2}pq}{a_{1}pq-Q(q-p)}\right)^{\frac{Q(q-p)-a_{2}pq}{(a_{1}-a_{2})p^{2}}}
\left(\int_{\mathbb{G}}|\mathcal{R}_{2}^{\frac{a_{2}}{\nu_{2}}}\phi(x)|^{p}dx\right)^{\frac{q-p}{p}}.$$
Since $u$ is chosen arbitrarily, we arrive at
\begin{equation}\label{sharp_thm5}C_{GN, \R_{1}, \R_{2}}^{-1}\geq\frac{a_{1}pq-Q(q-p)}{(a_{1}-a_{2})pq}\left(\frac{Q(q-p)-a_{2}pq}{a_{1}pq-Q(q-p)}\right)^{\frac{Q(q-p)-a_{2}pq}{(a_{1}-a_{2})p^{2}}}
\left(\int_{\mathbb{G}}|\mathcal{R}_{2}^{\frac{a_{2}}{\nu_{2}}}\phi(x)|^{p}dx\right)^{\frac{q-p}{p}}.
\end{equation}
Thus, from \eqref{sharp_thm001} and \eqref{sharp_thm5}, we obtain the first equality in \eqref{sharp1}.

Finally, the second equality in \eqref{sharp1} follows from the first equality in \eqref{sharp1} and \eqref{sharp_d}.
\end{proof}

\section{Best constants in the Sobolev inequalities}
\label{SEC:sharp2}

In this section, we investigate the constant $C_{S, \mathcal{R}}=C_{S, \mathcal{R}, a, p, q}$, which is the smallest positive constant in the following Sobolev inequality
\begin{equation}\label{Sobolev}
\left(\int_{\G}|u(x)|^{q}dx\right)^{\frac{p}{q}}\leq C\int_{\G}(|\mathcal{R}^{\frac{a}{\nu}}u(x)|^{p}+|u(x)|^{p})dx,
\end{equation}
where $u\in L^{p}_{a}(\mathbb{G})$, which is the Sobolev embedding theorem, on graded Lie groups established in \cite[Theorem 4.4.28]{FR16}.
Therefore, we can write
\begin{equation}\label{Sobolev_sharp1}
C_{S, \R}^{-1}:=\inf_{u\in L^{p}_{a}(\mathbb{G})\backslash\{0\}}\frac{\int_{\G}(|\mathcal{R}^{\frac{a}{\nu}}u(x)|^{p}+|u(x)|^{p})dx}
{\left(\int_{\G}|u(x)|^{q}dx\right)^{\frac{p}{q}}}.
\end{equation}
\begin{thm}\label{Sobolev}
 Let $a>0$, $1<p<\frac{Q}{a}$ and $p<q<\frac{pQ}{Q-ap}$. Let $\phi$ be a least energy solution of \eqref{nonlinear} and let $C_{S, \R}$ be the smallest positive constant $C$ in \eqref{Sobolev}. Then we have
\begin{equation}\label{Sobolev_sharp2}
C_{S, \R}=\left(\frac{apq}{apq-Q(q-p)}\int_{\G}|\phi(x)|^{p}dx\right)^{\frac{p-q}{q}}=
\left(\frac{pq}{q-p}d\right)^{\frac{p-q}{q}},
\end{equation}
where $d$ is defined in \eqref{d}.
\end{thm}
\begin{proof}[Proof of Theorem \ref{Sobolev}] Using Lemma \ref{sharp_lem} for $a_{1}=a$, $a_{2}=0$, $\R_{1}=\R$ and $\I(\phi)=0$ one calculates
$$\frac{\int_{\G}(|\mathcal{R}^{\frac{a}{\nu}}\phi(x)|^{p}+|\phi(x)|^{p})dx}
{\left(\int_{\G}|\phi(x)|^{q}dx\right)^{\frac{p}{q}}}=\left(\int_{\G}|\phi(x)|^{q}\right)^{\frac{q-p}{q}}=
\left(\frac{apq}{apq-Q(q-p)}\int_{\G}|\phi(x)|^{p}dx\right)^{\frac{q-p}{q}}.$$
It follows that
\begin{equation}\label{Sobolev_sharp3}
C_{S, \R}^{-1}\leq \left(\frac{apq}{apq-Q(q-p)}\int_{\G}|\phi(x)|^{p}dx\right)^{\frac{q-p}{q}}.
\end{equation}
For all $u\in L^{p}_{a}(\mathbb{G})\backslash\{0\}$, we denote $\widetilde{u}(x):=\|\phi\|_{L^{q}(\mathbb{G})}\|u\|^{-1}_{L^{q}(\G)}u(x)$ so that
$$\int_{\G}|\widetilde{u}(x)|^{q}dx=\int_{\G}|\phi(x)|^{q}dx.$$
Then, by Lemma \ref{sharp_lem2} for $a_{1}=a$, $a_{2}=0$ and $\R_{1}=\R$ we obtain that
\begin{multline*}
\int_{\G}(|\mathcal{R}^{\frac{a}{\nu}}\widetilde{u}(x)|^{p}+|\widetilde{u}(x)|^{p})dx\geq \int_{\G}(|\mathcal{R}^{\frac{a}{\nu}}\phi(x)|^{p}+|\phi(x)|^{p})dx
\\
=\frac{apq}{apq-Q(q-p)}\int_{\G}|\phi(x)|^{p}dx,
\end{multline*}
the last equality holds in view of \eqref{sharp_lem1} for $a_{1}=a$, $a_{2}=0$ and $\R_{1}=\R$.
In the case $a_{1}=a$ and $a_{2}=0$, by \eqref{sharp_lem2} we obtain
$$\frac{\int_{\G}(|\mathcal{R}^{\frac{a}{\nu}}u(x)|^{p}+|u(x)|^{p})dx}
{\left(\int_{\G}|u(x)|^{q}dx\right)^{\frac{p}{q}}}\geq \frac{apq}{apq-Q(q-p)}\int_{\G}|\phi(x)|^{p}dx \left(\int_{\G}|\phi(x)|^{q}dx\right)^{-\frac{p}{q}}$$
\begin{equation}\label{Sobolev_sharp4}
=\left(\frac{apq}{apq-Q(q-p)}\int_{\G}|\phi(x)|^{p}dx\right)^{1-\frac{p}{q}}.
\end{equation}
Thus, the estimates \eqref{Sobolev_sharp3} and \eqref{Sobolev_sharp4} imply the first equality in \eqref{Sobolev_sharp2}. Putting \eqref{sharp_d} in the first equality of \eqref{Sobolev_sharp2} for $a_{1}=a$ and $a_{2}=0$, we obtain the second equality in \eqref{Sobolev_sharp2}.
\end{proof}
\begin{rem}\label{Sobolev_rem}
From  Theorem \ref{sharp} and Theorem \ref{Sobolev}, we note that the best constants in the Sobolev and Gagliardo-Nirenberg (with $a_2=0$) inequalities are related by
$$C_{GN, \R, a, 0}^{\frac{p}{q}}=C_{S, \R}\frac{apq}{apq-Q(q-p)}\left(\frac{Q(q-p)}{apq-Q(q-p)}\right)^{\frac{Q(p-q)}{apq}}.$$
It is interesting to note that while each best constant depends on the positive Rockland operator $\R$ used in the definition of the norm, the ratio $C_{GN, \R, a, 0}^{\frac{p}{q}}/C_{S, \R}$ is independent of $\R$.
\end{rem}

\section{Further extensions}
\label{SEC:ex_exist}

In this section, we prove a more general form of the Gagliardo-Nirenberg inequality on graded Lie groups involving more than two norms, and using this one can prove the extension of the Theorem \ref{thm1} on ground state solutions in almost exactly the same way as in the proof of that theorem. Therefore, we will show this only briefly indicating the main changes in the proofs.

Here we use the space
$L^{p}_{a_{1},...,a_{\ell}}(\mathbb{G})=L^{p}_{a_{1},...,a_{\ell},\R_{1},...,\R_{\ell}}(\mathbb{G})$, for $a_{1}>a_{2}>...>a_{\ell}\geq0$, defined by the norm
\begin{equation}\label{norm_0}
\|u\|_{L^{p}_{a_{1},...,a_{\ell}, \R_{1},...,\R_{\ell}}(\mathbb{G})}:=
\left(\int_{\mathbb{G}}(\sum_{j=1}^{\ell}|\mathcal{R}_{j}^{\frac{a_{j}}{\nu_{j}}}u(x)|^{p})dx\right)^{1/p}.
\end{equation}
Note that $L^{p}_{a_{1},0,...,0}(\G)=L^{p}_{a_{1}}(\G)$ and $L^{p}_{a_{1},a_{2},0,...,0}(\G)=L^{p}_{a_{1},a_{2}}(\G)$. As in Remark \ref{rem_ind}, we obtain the independence of the spaces $L^{p}_{a_{1},...,a_{\ell}}(\G)$ from the choice of the Rockland operators $\{\R_{j}\}_{j=1}^{\ell}$, also in view of the equality
$$
L^{p}_{a_{1},...,a_{\ell}}(\G)=\bigcap_{j=1}^\ell \dot{L}^p_{a_j}(\G).
$$
Then we have the following multi-parameter Gagliardo-Nirenberg type inequality extending Theorem \ref{THM-GN}:
\begin{thm}\label{THM-GN_ex}
Let $\mathbb{G}$ be a graded Lie group of homogeneous dimension $Q$ and let $\{\mathcal{R}_{j}\}_{j=1}^{\ell}$ be positive Rockland operators of homogeneous degrees $\{\nu_{j}\}_{j=1}^{\ell}$, respectively, with any $\ell\geq 2$. Assume that
\begin{equation}\label{con}
a_{1}>a_{2}>...>a_{\ell}\geq0,\; 1<p<\frac{Q}{a_{1}},\; \textrm{ and }  \;Q\left(\frac{1}{p}-\frac{1}{p_{j}}\right)=a_{j},\; j=1\ldots,\ell,
\end{equation}
and that for $s_{j}\in[0,1]$ and $1<q<\infty$ we have
\begin{equation}\label{con2}
\sum_{j=1}^{\ell}s_{j}=1\; \textrm{ and } \;
\sum_{j=1}^{\ell}\frac{s_{j}}{p_{j}}=\frac1q.
\end{equation}
Then there exists $C>0$ such that we have
\begin{equation}\label{GN1_ex}
\|u\|_{L^{q}(\G)}\leq C\prod_{j=1}^{\ell}\left\|u\right\|^{s_{j}}_{\dot{L}^{p}_{a_j}(\G)}
\simeq C\prod_{j=1}^{\ell}\left\|\R_{j}^{\frac{a_{j}}{\nu_{j}}}u\right\|^{s_{j}}_{L^{p}(\G)}
\end{equation}
for any $u\in \bigcap_{j=1}^\ell \dot{L}^p_{a_j}(\G)$.
\end{thm}
\begin{proof}[Proof of Theorem \ref{THM-GN_ex}] Taking into account that $\sum_{j=1}^{\ell}s_{j}=1$, by a direct calculation one has
$$\int_{\G}|u(x)|^{q}dx=\int_{\G}|u(x)|^{q\left(\sum_{j=1}^{\ell}s_{j}\right)}dx
\leq \prod_{j=1}^{\ell}\left(\int_{\G}|u(x)|^{p_{j}}dx\right)^{\frac{qs_{j}}{p_{j}}},$$
where we have used the H\"{o}lder's inequality for $\sum_{j=1}^{\ell}\frac{qs_{j}}{p_{j}}=1$ in the last inequality.
Since $Q\left(\frac{1}{p}-\frac{1}{p_{j}}\right)=a_{j}$ for all $j=1,\ldots,\ell$, we use the Sobolev inequality \eqref{cor_ineq2} to obtain
$$\int_{\G}|u(x)|^{q}dx \leq C\prod_{j=1}^{\ell}\left(\int_{\G}|\R_{j}^{\frac{a_{j}}{\nu_{j}}}u(x)|^{p}dx\right)^{\frac{qs_{j}}{p}},$$
which implies \eqref{GN1_ex}.
\end{proof}
We now consider the following Schr\"{o}dinger equation with the power nonlinearity
\begin{equation}\label{nonlinear_ex}
\sum_{j=1}^{\ell}\mathcal{R}_{j}^{\frac{a_{j}}{\nu_{j}}}
(|\mathcal{R}_{j}^{\frac{a_{j}}{\nu_{j}}}u|^{p-2}\mathcal{R}_{j}^{\frac{a_{j}}{\nu_{j}}}u)=|u|^{q-2}u, \quad u\in \bigcap_{j=1}^\ell \dot{L}^p_{a_j}(\G).
\end{equation}
As an example in $\Rn$, for $p=2$ and $\R_j=(-1)^{m_j}\sum_{k=1}^n a_{jk}\frac{\partial^{2m_j}}{\partial x_k^{2m_j}}$, $a_{jk}>0$, $m_j\in\mathbb N_0$, we can consider the higher order partial differential equation with lower order terms:
\begin{equation}\label{EQ:hoex}
\sum_{j=1}^\ell \left( (-1)^{m_j}\sum_{k=1}^n a_{jk}\frac{\partial^{2m_j} u}{\partial x_k^{2m_j}}\right)=|u|^{q-2}u,
\end{equation}
or similar equations e.g. on the Heisenberg, or on more general stratified and graded Lie groups.
Now let us formulate some notations and definitions.

\begin{defn}
A function $u\in L^{p}_{a_{1},...,a_{\ell}}(\mathbb{G})$ is said to be a solution of \eqref{nonlinear_ex} if and only if for all $\psi\in L^{p}_{a_{1},...,a_{\ell}}(\mathbb{G})$ the identity
\begin{equation}\label{solution_ex}
\int_{\mathbb{G}}\left(\sum_{j=1}^{\ell}|\mathcal{R}_{j}^{\frac{a_{j}}{\nu_{j}}}
u(x)|^{p-2}\mathcal{R}_{j}^{\frac{a_{j}}{\nu_{j}}}
u(x)\overline{\mathcal{R}_{j}^{\frac{a_{j}}{\nu_{j}}}\psi(x)}
-|u(x)|^{q-2}u(x)\overline{\psi(x)}\right)dx=0
\end{equation}
holds.
\end{defn}

By $\mathfrak{L}_{1}:L^{p}_{a_{1},...,a_{\ell}}(\mathbb{G})\to \mathbb R$ and $\mathfrak{I}_{1}:L^{p}_{a_{1},...,a_{\ell}}(\mathbb{G})\to \mathbb R$ we define the following functionals acting on $L^{p}_{a_{1},...,a_{\ell}}(\mathbb{G})\cap L^{q}(\G)$:
\begin{equation}\label{L_ex}
\mathfrak{L}_{1}(u):=\frac{1}{p}\sum_{j=1}^{\ell}\int\limits_{\mathbb{G}}|\mathcal{R}_{j}^{\frac{a_{j}}{\nu_{j}}}u(x)|^{p}dx
-\frac{1}{q}\int\limits_{\mathbb{G}}|u(x)|^{q}dx
\end{equation}
and
\begin{equation}\label{I_ex}
\mathfrak{I}_{1}(u):=\int\limits_{\mathbb{G}}\left(\sum_{j=1}^{\ell}|\mathcal{R}_{j}^{\frac{a_{j}}{\nu_{j}}}u(x)|^{p}
-|u(x)|^{q}\right)dx.
\end{equation}

The Nehari set and $d_{1}$ are defined as
\begin{equation}\label{N_ex}
\mathcal{N}_{1}:=\{u\in L^{p}_{a_{1},...,a_{\ell}}(\mathbb{G})\ \backslash\{0\}: \I_{1}(u)=0\}
\end{equation}
and
\begin{equation}\label{d_ex}
d_{1}:=\inf\{\L_{1}(u):u\in\mathcal{N}_{1}\},
\end{equation}
respectively.
Now we state the extension of  Theorem \ref{thm1}.
\begin{thm}\label{thm1_ex}
Let conditions \eqref{con} and \eqref{con2} hold. Then the equation \eqref{nonlinear_ex} has a least energy solution $\phi\in \bigcap_{j=1}^\ell \dot{L}^p_{a_j}(\G).$

Moreover, we have $d_1=\L_1(\phi)$.
\end{thm}
The proof is similar to that of Theorem \ref{thm1} so we may omit the details.

We also note that there is a link between the best constant in the multi-parameter Gagliardo-Nirenberg inequality \eqref{GN1_ex}, the least energy solution $\phi$ from Theorem \ref{thm1_ex}, and the minimiser $d_1$ from \eqref{d_ex}. This can be derived similarly to the proof of Theorem \ref{sharp}, based on relations for Sobolev norms of $\phi$ similar to those in Lemma \ref{sharp_lem}. However, for more than two operators such relations are more lengthy, therefore we omit such a calculation here.


\end{document}